\documentclass[12pt]{article}

\usepackage{amssymb}
\usepackage{amsmath}
\usepackage{amsfonts}
\usepackage{mitpress}

\newtheorem{theorem}{Theorem}

\newtheorem{corollary}[theorem]{Corollary}

\newtheorem{example}[theorem]{Example}

\newtheorem{lemma}[theorem]{Lemma}

\newtheorem{proposition}[theorem]{Proposition}
\newtheorem{remark}[theorem]{Remark}

\newenvironment{proof}[1][Proof]{\noindent\textbf{#1.} }{\ \rule{0.5em}{0.5em}}
\newdimen\dummy
\dummy=\oddsidemargin
\begin{document}

\title{A Gibbs Conditional theorem under extreme deviation}
\author{Maeva Biret$^{(1)}$, Michel Broniatowski$^{(2,\ast )}$, Zangsheng Cao$^{(2)}$ \\
$^{(1)}$ SNECMA***, $^{(2)}$ LSTA, Universit\'{e} Pierre et Marie Curie, Paris, France, $^{(\ast )}$ Corresponding author}
\maketitle

\begin{abstract}
We  explore some properties of the conditional distribution of an  i.i.d. sample under large exceedances of its sum. Thresholds for the asymptotic independance of the summands are observed, in contrast with the classical case when the conditioning event is in the range of a large deviation. This paper is an extension to \cite{BoniaCao}. Tools include a new Edgeworth expansion adapted to specific triangular arrays where the rows are generated by tilted distribution with diverging parameters, together with some Abelian type results.
\end{abstract}

\section{Introduction}

\label{secintroduction}

Let $X_{1}^{n}:=\left( X_{1},..,X_{n}\right)$ be $n$ independent unbounded real valued random variables and $S_{1}^{n}:=X_{1}+..+X_{n}$ denote their sum. The purpose of this paper is to explore the limit distribution of the generic variable $X_{1}$ conditioned on extreme deviations (ED) pertaining to $S_{1}^{n}.$ By extreme deviation we mean that $S_{1}^{n}/n$ is supposed to take values which are going to infinity as $n$ increases. Obviously such events are of infinitesimal probability. Our interest in this question stems from a first result which assesses that under appropriate conditions, when the sequence $a_{n}$ is such that 
\[
\lim_{n\rightarrow \infty }a_{n}=\infty
\]
then there exists a sequence $\varepsilon_{n}$ which tends to $0$ as $n$ tends to infinity such that 
\begin{equation}
\lim_{n\rightarrow\infty}P\left(\left.\cap_{i=1}^{n}\left(X_{i}\in
\left(a_{n}-\varepsilon_{n},a_{n}+\varepsilon_{n}\right)\right)
\right\vert S_{1}^{n}/n>a_{n}\right)=1,  \label{democracy}
\end{equation}
which is to say that when the empirical mean takes exceedingly large values, then all the summands share the same behaviour. This result obviously requires a number of hypotheses, which we simply quote as "light tails" type. We refer to \cite{BoniaCao} for this result and the connection with earlier related works; that such most unusual cases may be considered is argumented in this latest paper, in relation with the Erd\"{o}s-R\'{e}niy law of large numbers and the formation of high level aggregates in random sequences.

The above result is clearly to be put in relation with the so-called Gibbs conditional Principle which we recall briefly in its simplest form.

Consider the case when the sequence $a_{n}$ $=a$ is constant with value larger than the expectation of $X_{1}.$ Hence we consider the behaviour of the summands when $\left(S_{1}^{n}/n>a\right)$, under a large deviation (LD) condition about the empirical mean. The asymptotic conditional distribution of $X_{1}$ given $\left(S_{1}^{n}/n>a\right)$ is the well known tilted distribution of $P_{X}$ with parameter $t$ associated to $a$. Let us introduce some notation to shed some light on this The hypotheses to be stated now together with notation are kept throughout the entire paper. Without loss of generality it is assumed that the generic r.v. $X_{1}$ takes only non negative values.

It will be assumed that $P_{X}$ , which is the distribution of $X_{1}$, has a density $p$ with respect to the Lebesgue measure on $\mathbb{R}$. The fact that $X_{1}$ has a light tail is captured in the hypothesis that $X_{1}$ has a moment generating function 
\[
\Phi(t):=E\lbrack\exp tX_{1}\rbrack,
\]
which is finite in a non void neighborhood $\mathcal{N}$ of $0$. This fact is usually refered to as a Cramer type condition.

Defined on $\mathcal{N}$ are the following functions. The functions 
\begin{equation}
t\rightarrow m(t):=\frac{d}{dt}\log\Phi(t)  \label{m}
\end{equation}
\begin{equation}
t\rightarrow s^{2}(t):=\frac{d}{dt}m(t)  \label{s^2}
\end{equation}

\begin{equation}
t\rightarrow\mu_{j}(t):=\frac{d^{j}}{dt^{j}}\log\Phi(t)\text{, }
j\geq 3  \label{mu_j}
\end{equation}
are the expectation, the variance, and the centered moments of order $j$ of the r.v. $\mathcal{X}_{t}$ with density 
\[
\pi_{t}(x):=\frac{\exp tx}{\Phi(t)}p(x)
\]
which is defined on $\mathbb{R}$ and which is the tilted density with
parameter $t$. When $\Phi$ is steep, meaning that 
\[
\lim_{t\rightarrow t^{+}}m(t)=\infty
\]
where $t^{+}:=ess\sup\mathcal{N}$ then $m$ parametrizes the convex hull of the support of $P_{X}$. We refer to Barndorff-Nielsen \cite{barndorff} for those properties. As a consequence of this fact, for all $a$ in the support of $P_{X}$, it will be convenient to define 
\[
\pi^{a}=\pi_{t}
\]
where $a$ is the unique solution of the equation $m(t)=a$.

The Gibbs conditional principle in the standard above setting can be stated as follows.

As $n$ tends to infinity the conditional distribution of $X_{1}$ given $\left(S_{1}^{n}/n>a\right)$ is $\Pi^{a}$, the distribution with density $\pi^{a}$.

Indeed we prefer to state Gibbs principle in a form where the conditioning event is a point condition $\left(S_{1}^{n}/n=a\right) $. The conditional distribution of $X_{1}$ given $\left(S_{1}^{n}/n=a\right)$ is a well defined distribution and Gibbs conditional principle states that this conditional distribution converges to $\Pi^{a}$ as $n$ tends to infinity. In both settings, this convergence holds in total variation norm. We refer to \cite{Diaconis1} for the local form of the conditioning event; we will mostly be interested in the extension of this form in the present paper.

For all $\alpha$ (depending on $n$ or not) we will denote $p_{\alpha}$ the density of the random vector $X_{1}^{k}$ conditioned upon the local event $\left(S_{1}^{n}=n\alpha\right)$. The notation $p_{\alpha}\left(X_{1}^{k}=x_{1}^{k}\right)$ is sometimes used to denote the value of the density $p_{\alpha}$ at point $x_{1}^{k}.$ The same notation is used when $X_{1},\ldots,X_{k}$ are sampled under some $\Pi^{\alpha}$, namely $\pi^{\alpha}(X_{1}^{k}=x_{1}^{k})$.

This article is organized as follows. Notation and hypotheses are stated in Section \ref{notation}, along with some necessary facts from asymptotic analysis in the context of light tailed densities. Section \ref{secGibbs} provides a local Gibbs conditional principle under EDP, namely producing the approximation of the conditional density of $X_{1}$ conditionally on $\left(S_{1}^{n}/n=a_{n}\right)$ for sequences $a_{n}$ which tend to infinity. We explore two rates of growth for the sequence $a_{n}$, which yield two different approximating distributions for the conditional law of $X_{1}$. The first one extends the classical approximation by the tilted one,
substituting $\pi^{a}$ by $\pi^{a_{n}}$. The second case, which corresponds to a faster growth of $a_{n}$, produces an approximation of a different kind. It may be possible to explore faster growth conditions than those considered here, leading to a wide class of approximating distributions; this would require some high order Edgeworth expansions for triangular arrays of variables, extending the corresponding result of order $3$ presented in this paper; we did not move further in this direction, in order to avoid too many technicalities.\bigskip

For fixed $k$ and fixed $a_{n}=a>E(X_{1})$ it is known that the r.v's $X_{1},\ldots,X_{k}$ are asymptotically independent given $\left( S_{1}^{n}/n=a_{n}\right)$; see \cite{Diaconis1}. This statement is explored when $a_{n}$ grows to infinity with $n$, keeping $k$ fixed. It is shown that the asymtotic independence property holds for sequences $a_{n}$ with moderate growth, and that independence fails for sequences $a_{n}$ with fast growth.

The local approximation of the density of $X_{1}$ conditionally on $\left(S_{1}^{n}/n=a_{n}\right) $ is further extended to typical paths under the conditional sampling scheme, which in turn provides the approximation in variation norm for the conditional distribution; the method used here follows closely the approach by \cite{Bronia}. The differences between the Gibbs principles in LDP and EDP are discussed. Section \ref{secExceedance} states similar results in the case when the conditioning event is $\left(S_{1}^{n}/n>a_{n}\right)$.

The main tools to be used come from asymptotic analysis and local limit theorems, developped from \cite{Feller} and \cite{Bingham}; we also have borrowed a number of arguments from \cite{Nagaev}. An Edgeworth expansion for some special array of independent r.v's with tilted distribution and argument moving to infinity with the row-size is needed; its proof is differed to the Section \ref{appendix}. The basic Abelian type result which is used is stated in \cite{BirBroCao}.

\section{Notation and hypotheses}\label{notation}
Thereafter we will use indifferently the notation $f(t)\underset{t\rightarrow\infty}{\sim}g(t)$ and $f(t)\underset{t\rightarrow \infty}{=}g(t)(1+o(1))$ to specify that $f$ and $g$ are asymptotically equivalent functions.\newline

The density $p$ is assumed to be of the form 
\begin{equation}
p(x)=\exp(-(g(x)-q(x))),\hspace{0.4cm}x\in\mathbb{R}_{+}.  \label{2.1}
\end{equation}

The function $q$ is assumed to be bounded, so that the asymptotic behaviour of $p$ is captured through the function $g$. The function $g$ is positive, convex, four times differentiable and satisfies 
\begin{equation}
\frac{g(x)}{x}\underset{x\rightarrow\infty}{\longrightarrow}\infty.
\label{2.2}
\end{equation}
Define 
\begin{equation}
h(x):=g^{\prime}(x). \label{h}
\end{equation}
In the present context, due to (\ref{2.2}) and the assumed conditions on $q$ to be stated hereunder, $t^{+}=+\infty $.

Not all positive convex $g$'s satisfying (\ref{2.2}) are adapted to our purpose. We follow the line of Juszczak and Nagaev \cite{Nagaev} to describe the assumed regularity conditions of $h$. See also \cite{BalKluRes1993} for somehow similar conditions.

We firstly assume that the function $h$, which is a positive function defined on $\mathbb{R}_{+}$, is either regularly or rapidly varying in a neighborhood of infinity; the function $h$ is monotone and, by (\ref{2.2}), $h(x)\rightarrow\infty$ when $x\rightarrow\infty$.

The following notation is adopted.

$RV(\alpha)$ designates the class of regularly varying functions of index $\alpha$ defined on $\mathbb{R}_{+}$,
\[
\psi(t):=h^{\leftarrow}(t)
\]
designates the inverse of $h.$ Hence $\psi$ is monotone for large $t$ and $\psi(t)\rightarrow\infty$ when $t\rightarrow\infty$, $\sigma^{2}(x):=1/h^{\prime}(x)$, $\hat{x}:=\hat{x}(t)=\psi(t)$, $\hat{\sigma}:=\sigma(\hat{x})=\sigma(\psi(t))$.

The two cases considered for $h$, the regularly varying case and the rapidly varying case, are described below. The first one is adapted to regularly varying functions $g$, whose smoothness is described through the following condition pertaining to $h$.\newline
\label{DEF2.1} \textbf{The Regularly varying case.} It will be assumed that $h$ belongs to the subclass of $RV(\beta)$, $\beta>0$, with 
\[
h(x)=x^{\beta }l(x),
\]
where the Karamata form of the slowly varying function $l$ takes the form 
\begin{equation}
l(x)=c\exp \int_{1}^{x}\frac{\epsilon (u)}{u}du  \label{Karamata1}
\end{equation}
for some positive $c$. We assume that $x\mapsto \epsilon (x)$ is twice differentiable and satisfies 
\begin{equation}
\left\{ 
\begin{array}{lll}
& \epsilon(x)\underset{x\rightarrow\infty}{=}o(1), &  \\ 
& x|\epsilon^{\prime}(x)|\underset{x\rightarrow\infty}{=}O(1), &  \\ 
& x^{2}|\epsilon^{(2)}(x)|\underset{x\rightarrow\infty}{=}O(1). & 
\end{array}
\right.  \label{2.3}
\end{equation}
It will also be assumed that 
\begin{equation}
|h^{(2)}(x)|\in RV(\theta) \label{h^(2)}
\end{equation}
where $\theta$ is a real number such that $\theta\leq\beta -2$.

\begin{remark}
\label{REM2.1} Under (\ref{Karamata1}), when $\beta\not=1$ then, under (\ref{h^(2)}), $\theta=\beta-2$. Whereas, when $\beta=1$ then $\theta\leq\beta -2$. A sufficient condition for the last assumption (\ref{h^(2)}) is that $\epsilon^{\prime}(t)\in RV(\gamma)$, for some $\gamma<-1$. Also in this case when $\beta=1$, then $\theta=\beta+\gamma-1$.
\end{remark}

\begin{example}
\label{EX2.1} {\textbf{\emph{Weibull density.}}} Let $p$ be a Weibull density with shape parameter $k>1$ and scale parameter 1, namely 
\begin{align*}
p(x)& =kx^{k-1}\exp(-x^{k}),\hspace{1cm}x\geq 0 \\
& =k\exp(-(x^{k}-(k-1)\log x)).
\end{align*}
Take $g(x)=x^{k}-(k-1)\log x$ and $q(x)=0$. Then it holds 
\[
h(x)=kx^{k-1}-\frac{k-1}{x}=x^{k-1}\left(k-\frac{k-1}{x^{k}}\right).
\]
Set $l(x)=k-(k-1)/x^{k},x\geq 1$, which verifies 
\[
l^{\prime}(x)=\frac{k(k-1)}{x^{k+1}}=\frac{l(x)\epsilon(x)}{x}
\]
with 
\[
\epsilon(x)=\frac{k(k-1)}{kx^{k}-(k-1)}.
\]
Since the function $\epsilon(x)$ satisfies the three conditions in (\ref{2.3}), then $h(x)\in RV(k-1)$.
\end{example}

\label{DEF2.2} \textbf{The Rapidly varying case.} Here we have $h^{\leftarrow}(t)=\psi(t)\in RV(0)$ and 
\begin{equation}
\psi(t)=c\exp\int_{1}^{t}\frac{\epsilon(u)}{u}du,  \label{Karamata 2}
\end{equation}
for some positive $c$, and $t\mapsto\epsilon(t)$ is twice differentiable with 
\begin{equation}
\left\{ 
\begin{array}{lll}
& \epsilon(t)\underset{t\rightarrow\infty}{=}o(1), &  \\ 
& \frac{t\epsilon^{\prime}(t)}{\epsilon(t)}\underset{t\rightarrow \infty}{\longrightarrow }0, &  \\ 
& \frac{t^{2}\epsilon^{(2)}(t)}{\epsilon(t)}\underset{t\rightarrow \infty }{\longrightarrow}0. & 
\end{array}
\right.  \label{2.4}
\end{equation}
Note that these assumptions imply that $\epsilon(t)\in RV(0)$.

\begin{example}
\label{EX2.2} {\textbf{\emph{A rapidly varying density.}}} Define $p$ through 
\[
p(x)=c\exp(-e^{x-1}),x\geq 0.
\]
Then $g(x)=h(x)=e^{x-1}$ and $q(x)=0$ for all non negative $x$. We show that $h(x)$ is a rapidly varying function. It holds $\psi(t)=\log t+1$. Since $\psi^{\prime}(t)=1/t$, let $\epsilon(t)=1/(\log t+1)$ so that $\psi^{\prime}(t)=\psi(t)\epsilon(t)/t$. Moreover, the three conditions of (\ref{2.4}) are satisfied. Thus $\psi(t)\in RV(0)$ and $h(x)$ is a rapidly varying function.
\end{example}

Denote by $\mathcal{R}$ the class of functions with either regular variation defined as in Case \ref{DEF2.1} or with rapid variation defined as in Case \ref{DEF2.2}.

We now state hypotheses pertaining to the bounded function $q$ in (\ref{2.1}). We assume that 
\begin{equation}
|q(x)|\in RV(\eta),\mbox{ for some $\eta<\theta-\frac{3\beta}{2}-\frac{3}{2}$ if $h\in RV(\beta)$}  \label{2.5}
\end{equation}
and 
\begin{equation}
|q(\psi(t))|\in RV(\eta),
\mbox{ for some $\eta<-\frac{1}{2}$ if $h$ is rapidly varying.}  \label{2.6}
\end{equation}

We will make use of the following result (see \cite{BirBroCao} Thm 3.1).

\begin{theorem}
\label{THM3.1}Let $p(x)$ be defined as in (\ref{2.1}) and $h(x)$ belong to $\mathcal{R}$. Denote by $m(t)$, $s^{2}(t)$ and $\mu_{j}(t)$ for $j=3,4,...$ the functions defined in (\ref{m}), (\ref{s^2}) and (\ref{mu_j}). Then it holds 
\begin{align*}
m(t)& \underset{t\rightarrow\infty}{=}\psi(t)(1+o(1)), \\
s^{2}(t)& \underset{t\rightarrow\infty}{=}\psi^{\prime}(t)(1+o(1)), \\
\mu_{3}(t)& \underset{t\rightarrow\infty}{=}\psi^{(2)}(t)(1+o(1)), \\
\mu_{j}(t)& \underset{t\rightarrow\infty}{=}\left\{ 
\begin{array}{ll}
M_{j}s^{j}(t)(1+o(1)), & \mbox{for even $j>3$} \\ 
\frac{(M_{j+3}-3jM_{j-1})\mu_{3}(t)s^{j-3}(t)}{6}(1+o(1)), & \mbox{for odd $j>3$}
\end{array}
\right. ,
\end{align*}
where $M_{i}$, $i>0$, denotes the $i$th order moment of standard normal distribution.
\end{theorem}

\begin{corollary}
\label{3cor1} Let $p(x)$ be defined as in (\ref{2.1}) and $h(x)\in \mathfrak{R}$. Then it holds as $t\rightarrow \infty $ 
\[
\frac{\mu _{3}(t)}{s^{3}(t)}\longrightarrow 0.
\]
\end{corollary}

\begin{proof}
In the regularly varying case this follows from Corollaries 1 and 2 in \cite{BirBroCao}, and in the rapidly varying case from Corollary 3 and Lemma 3 in \cite{BirBroCao}.
\end{proof}

Our results require an extension of the classical Edgeworth expansions to triangular arrays of row-wise independent and identically distributed random variables, where the expectation of the generic r.v. in the $n-$th row tends to infinity with $n.$ This can be achieved under log-concavity of $p$, i.e. when the function $q$ is the null function, or when $p$ is nearly log-concave. This is the scope of the next Section.

\section{Edgeworth expansion under extreme normalizing factors}

With $\pi^{a_{n}}$ defined through 
\[
\pi^{a_{n}}(x)=\frac{e^{tx}p(x)}{\Phi(t)},
\]
and $t$ determined by $m(t)=a_{n}$ together with $s^{2}:=s^{2}(t)$ define the normalized density of $\pi^{a_{n}}$ by 
\[
\bar{\pi}^{a_{n}}(x)=s\pi^{a_{n}}(sx+a_{n}),
\]
and denote the $n$-convolution of $\bar{\pi}^{a_{n}}(x)$ by $\bar{\pi}_{n}^{a_{n}}(x)$. Denote by $\rho_{n}$ its normalized density 
\[
\rho_{n}(x):=\sqrt{n}\bar{\pi}_{n}^{a_{n}}(\sqrt{n}x).
\]
The following result extends the local Edgeworth expansion of the distribution of normalized sums of i.i.d. r.v.'s to the present context, where the summands are generated under the density $\bar{\pi}^{a_{n}}$.
Therefore the setting is that of a triangular array of rowwise independent summands; the fact that $a_{n}\rightarrow\infty $ makes the situation unusual. We mainly adapt Feller's proof (Chapiter 16, Theorem 2 \cite{Feller}). However this variation on the classical Edgeworth expansion result requires some additional regularity assumption, which meet the requirements of Theorem \ref{THM3.1}, which are fulfilled in most models dealing with extremes and convolutions. Those are captured in cases when the density $p$ is log-concave, or nearly log concave in the upper tail. Similar conditions are considered in \cite{BoniaCao}.

\begin{theorem}
\label{3theorem1} With the above notation, uniformly upon $x$ it holds 
\[
\rho_{n}(x)=\phi (x)\Big(1+\frac{\mu_{3}}{6\sqrt{n}s^{3}}\big(x^{3}-3x\big)\Big)+o\Big(\frac{1}{\sqrt{n}}\Big).
\]
where $\phi (x)$ is standard normal density.
\end{theorem}

The proof of this result is postponed to the Section \ref{appendix}.

\section{Gibbs' conditional principles under extreme events}\label{secGibbs}
We now explore Gibbs conditional principles under extreme events. The first result is a pointwise approximation of the conditional density $p_{a_{n}}\left(y_{1}\right)$ on $\mathbb{R}.$ Two cases will be considered according to the rate of growth of the sequence $a_{n}$ to infinity. For "moderate" growth our result extends the classical one pertaining to constant $a_{n}$ larger than $E(X_{1})$, since the approximating density of $p_{a_{n}}$ is the tilted distribution with parameter $a_{n}$. For sequences $a_{n}$ with fast growth, the approximating density includes a second order term which contributes to the approximation in a similar role as the tilted term; this term also appears in the first case, but is negligible with respect to the tilted density.

However this local approximation can be greatly improved when comparing $p_{a_{n}}$ to its approximation. We will first prove that the approximation holds when the fixed arbitrary $y_{1}$ is substituted by a r.v. $Y_{1}$ with distribution $P_{a_{n}}$, henceforth on a typical realization under the distribution to be approximated. The approximation therefore holds in probability under this sampling scheme; a simple Lemma then proves that such a statement implies that the total variation distance between $P_{a_{n}}$ and its approximation tends to $0$ as $n$ tends to infinity.

As a by-product we also address similar approximations for the case when the conditioning event writes $\left(S_{1}^{n}/n>a_{n}\right) .$ The case when $a_{n}$ grows to infinitly fast enough overlaps with that for which (\ref{democracy}) holds.

Extension to the approximation of the distribution of $X_{1}$ given $\left(T_{n}=a_{n}\right)$ or $\left(T_{n}>a_{n}\right)$ where 
\[
T_{n}:=\frac{1}{n}\sum_{i=1}^{n}f(X_{i})
\]
for functions $f$ satisfying appropriate conditions are considered.

For sake of completeness we also provide some information when the density $p_{a_{n}}$ is that of the vector $\left(X_{1},\ldots,X_{k}\right) $ for fixed $k$. We prove that for moderate growth of $a_{n}$ the approximating density is the product of corresponding marginal approximations, generalizing the well known result by Csiszar \cite{Csiszar84} which, in the present context, assesses the limit conditional independence of the coordinates of the vector $\left(X_{1},\ldots,X_{k}\right)$ given $\left(S_{n}>na_{n}\right)$ for fixed $a_{n}>E(X_{1})$ and fixed $k.$ At the contrary this property is lost when $a_{n}$ grows quickly to infinity.

Because of the property (\ref{democracy}) it would be of interest to consider the joint distribution of the vector $\left(X_{1},\ldots,X_{k_{n}}\right)$ given $\left(S_{n}>na_{n}\right)$ for sequences $k_{n}$ close to $n$, as done in \cite{Bronia} for sequences $a_{n}$ ranging from CLT to LDP. The extreme deviation case adds noticeable analytical difficulties.

\subsection{A local result}
Fix $y_{1}^{k}:=\left(y_{1},\ldots,y_{k}\right) $ in $\mathbb{R}^{k}$ and define $s_{i}^{j}:=y_{i}+\ldots+y_{j}$ for $1\leq i<j\leq k$. Define $t$ through 
\begin{equation}
m:=m(t):=a_{n}  \label{3lfd01}
\end{equation}
and set 
\[
s:=s(t)
\]
for brevity.

We consider two conditions pertaining to the growth of the sequence $a_{n}$ to infinity. In the first case we assume that 
\begin{equation}
\lim_{n\rightarrow\infty}\frac{a_{n}}{s\sqrt{n}}=0,
\label{croissance de aen o(1)}
\end{equation}
and in the second case we consider sequences $a_{n}$ which may grow faster to infinity, obbeying 
\begin{equation}
0<\lim\inf_{n\rightarrow\infty}\frac{a_{n}}{s\sqrt{n}}\leq\lim\sup
_{n\rightarrow\infty}\frac{a_{n}}{s\sqrt{n}}<\infty.
\label{croissance de aen O(1)}
\end{equation}

\begin{remark}
Both conditions (\ref{croissance de aen o(1)}) and (\ref{croissance de aen O(1)}) can be expressed in terms of $a_{n}$ when the variance function $V(x)$ of the distribution of $X_{1}$ is known either in closed form, or is asymptotically equivalent to some known function.\ Recall that the variance function is defined on $Im(X_{1})$ through 
\[
x\rightarrow V(x)=s^{2}om^{-1}(x).
\]
See e.g. \cite{BarLev92} for a description of distribution functions with polynomial variance function and\ \cite{Jorgensen97} for tail equivalence for the variance function in infinitely divisible distributions. In the Regularly varying case, i.e. when $h$ belongs to the subclass of $RV(\beta)$, $\beta>0,$ then standard operations on smooth regularly varying functions yield $V(x)=x^{1-\beta }$ $l(x)$ for some slowly varying function $l$; see \cite{Bingham}; hence $s(t)=a_{n}^{\left(1-\beta\right)/2}l(a_{n})$. Assuming that $V(x)\thicksim x^{2\rho}$ as $x\rightarrow\infty$, it follows that (\ref{croissance de aen o(1)}) writes $a_{n}=o\left(n^{1/(1+\rho)}\right)$ whereas (\ref{croissance de aen O(1)}) amounts to assume that $a_{n}$ is of order $n^{1/(1+\rho)}$.
\end{remark}

Denote
\[
z:=\frac{m-y_{1}}{s\sqrt{n-1}}.
\]

\begin{theorem}
\label{pThm oint conditional density sous o(1)} When (\ref{croissance de aen o(1)}) holds then it holds 
\[
p_{a_{n}}(y_{1})=p(X_{1}=y_{1}|S_{1}^{n}=na_{n})=\pi^{a_{n}}(y_{1})\Big(1+o\big(\frac{1}{\sqrt{n}}\big)\Big),
\]
\end{theorem}

The proof of this result is postponed to the Section \ref{appendix}.

\begin{remark}
The above condition (\ref{croissance de aen o(1)}) is not sufficient to entail
(\ref{democracy}) to hold. This yields to the study of a similar limit conditional result under the corresponding condition (\ref{croissance de aen O(1)}).
\end{remark}

\begin{theorem}
\label{Thm point condit density sous O(1)}Assume that the sequence $a_{n}$ satisfies (\ref{croissance de aen O(1)}). Denote 
\[
\alpha:=t+\frac{\mu_{3}}{2(n-1)s^{2}}
\]
and 
\[
\beta:=(n-1)s^{2}.
\]
Then 
\begin{equation}
p(X_{1}=y_{1}|S_{1}^{n}=na_{n})=g_{a_{n}}(y_{1}):=Cp(y_{1})\mathfrak{n}\left(\alpha\beta+a_{n},\beta,y_{1}\right)(1+o(1))
\label{Thm-locSousO(1)}
\end{equation}
where $\mathfrak{n}\left(\mu,\sigma^{2},x\right)$ denotes the normal density with expectation $\mu $ and variance $\sigma^{2}$ evaluated at point $x$, and $C$ is a normalizing constant.

When (\ref{croissance de aen o(1)}) holds instead of (\ref{croissance de aen
O(1)}) then 
\[
p(X_{1}=y_{1}|S_{1}^{n}=na_{n})=\pi^{a_{n}}(y_{1})(1+o(1))
\]
for all $y_{1}$ as $n$ tends to infinity.
\end{theorem}

\begin{proof}
In contrast with the above case, the second summand in (\ref{conddens}) does not tend to $0$ any longer and contributes to the approximating density. Standard development then yields the result. When (\ref{croissance de aen o(1)}) holds instead of (\ref{croissance de aen O(1)}) then standard expansions in (\ref{Thm-locSousO(1)}) provide $g_{a_{n}}(y_{1})\sim\pi^{a_{n}}(y_{1})$ for all $y_{1}$ as $n$ tends to infinity.
\end{proof}

\subsection{On conditional independence under extreme events}

We now turn to the case when we approximate the joint conditional density $p_{a_{n}}(y_{1}^{k}):=p_{a_{n}}(y_{1},\ldots,y_{k})=p_{a_{n}}(\left.X_{1}^{k}=y_{1}^{k}\right\vert S_{1}^{n}=na_{n}).$ Denote $s_{i}^{j}:=y_{i}+\ldots+y_{j}$ for $i\leq j$ and $s_{1}^{0}:=0.$

We first consider the case when (\ref{croissance de aen o(1)}) holds. We then have

\begin{proposition}
\label{Prop indCondCroissance Lente}When (\ref{croissance de aen o(1)})
holds then for any fixed $k$ 
\[
p_{a_{n}}(y_{1}^{k})=\prod_{i=1}^{k}\pi^{m_{i}}(y_{i})\big(1+o(1/\sqrt{n})\big)
\]
where 
\begin{equation}
m_{i}:=m(t_{i}):=\frac{na_{n}-s_{1}^{i}}{n-i}.  \label{def-t_i}
\end{equation}
\end{proposition}

The proof of this result is postponed to the Section \ref{appendix}.

We now explore the limit conditional independence of blocks of fixed length under extreme condition. As a consequence of the above Proposition \ref{Prop indCondCroissance Lente} it holds

\begin{theorem}
\label{point conditional density e} Under (\ref{croissance de aen o(1)}) it holds 
\[
p_{a_{n}}(y_{1}^{k})=p(X_{1}^{k}=y_{1}^{k}|S_{1}^{n}=na_{n})=\left(1+o\left(\frac{1}{\sqrt{n}}\right)\right)\prod_{i=1}^{k}\pi^{a_{n}}(X_{i}=y_{i}),
\]
\end{theorem}

The technical proof is differed to the Section \ref{appendix}.

\begin{remark}
The above result shows that asymptotically the point condition $\left(S_{1}^{n}=na_{n}\right) $ leaves blocks of $k$ of the $X_{i}^{\prime }s$ independent. Obviously this property does not hold for large values of $k$, close to $n$. A similar statement holds in the LDP range, conditioning either on $\left(S_{1}^{n}=na\right)$ (see Diaconis and Friedman \cite{Diaconis1})), or on $\left(S_{1}^{n}\geq na\right)$ (see Csiszar \cite{Csiszar84} for a general statement on asymptotic conditional independence given events with positive probability).
\end{remark}

We now turn to the case when $a_{n}$ moves more quickly to infinity. Denote
\[
m_{i}:=m(t_{i}):=\frac{na_{n}-s_{1}^{i-1}}{n-i+1}
\]
together with 
\[
s_{i}^{2}:=s^{2}(t_{i}).
\]

\begin{theorem}
Assume that (\ref{croissance de aen O(1)}) holds. Then for all fixed $k$ it holds 
\[
p(X_{1}^{k}=y_{1}^{k}|S_{1}^{n}=na_{n})=\prod_{i=1}^{k}g_{i}(y_{i})\left(1+o\left( 1\right)\right)
\]
where 
\[
g_{i}(y_{i}):=C_{i}p(y_{i})\mathfrak{n}\left(\alpha_{i}\beta_{i}+a_{n},\beta_{i},y_{i}\right)
\]
and 
\begin{equation}
\alpha_{i}:=t_{i}+\frac{\mu_{3}}{2(n-i+1)s_{i}^{2}}  \label{alfa_i}
\end{equation}
\begin{equation}
\beta_{i}:=(n-i+1)s_{i}^{2}.  \label{beta_i}
\end{equation}
\end{theorem}

\begin{remark}
When (\ref{croissance de aen o(1)}) holds, the above result is a refinement of the
result in Proposition \ref{Prop indCondCroissance Lente}. Under (\ref{croissance de aen O(1)}) and when (\ref{croissance de aen o(1)}) does not hold, the approximations obtained in Lemma \ref{3lemma z} do not hold, and the approximating density cannot be stated as a product of densities under which independence holds. In that case it follows that the conditional independence property under extreme events does not hold any longer.
\end{remark}

\subsection{Strenghtening the local Gibbs conditional principle}

We now turn to a stronger approximation of $p_{a_{n}}$. Consider $Y_{1}$ a r.v. with density $p_{a_{n}}$, and the random variable $p_{a_{n}}\left(Y_{1}\right):=p(\left.X_{1}=Y_{1}\right\vert S_{1}^{n}=na_{n})$. Denote 
\[
g_{a_{n}}(x):=Cp(x)\mathfrak{n}\left(\alpha\beta+a_{n},\beta,x\right)
\]
where $\alpha:=\alpha_{n}$ and $\beta:=\beta_{n}$ are defined in (\ref{alfa_i}) and (\ref{beta_i}), and $C$ is a normalizing constant, it holds

\begin{theorem}
\label{ThmLocalProba}(i) When (\ref{croissance de aen o(1)}) holds then
\[
p_{a_{n}}\left(Y_{1}\right)=\pi^{a_{n}}\left(Y_{1}\right)\left(1+R_{n}\right)
\]
where the tilted density at point $a_{n}$, and where $R_{n}$ is a function of $Y_{1}$ such that $P_{a_{n}}\left(\left\vert R_{n}\right\vert>\delta\sqrt{n}\right)\rightarrow 0$ as $n\rightarrow \infty $ for any positive $\delta$. When (\ref{croissance de aen O(1)}) holds then, with $t_{n}$ such that $m(t_{n})=a_{n}$, $\alpha:=\alpha_{n}$ and $\beta:=\beta_{n}$ 
\[
p_{a_{n}}\left(Y_{1}\right)=g_{a_{n}}(Y_{1})(1+R_{n}^{\prime})
\]
where $P_{a_{n}}\left(\left\vert R_{n}^{\prime}\right\vert>\delta\right)\rightarrow 0$ as $n\rightarrow\infty$ for any positive $\delta$.
\end{theorem}

\begin{remark}
This result is of much greater relevance than the previous ones. Indeed under $P_{a_{n}}$ the r.v. $Y_{1}$ may take large values. On the contrary simple approximation of $p_{a_{n}}$ by $\pi^{a_{n}}$ or $g_{a_{n}}$ on $\mathbb{R}_{+}$ only provides some knowledge on $p_{a_{n}}$ on sets with smaller and smaller probability under $p_{a_{n}}$. Also it will be proved that as a consequence of the above result, the $L^{1}$ norm between $p_{a_{n}}$ and its approximation goes to $0$ as $n\rightarrow\infty$, a result out of reach through the aforementioned results.
\end{remark}

In order to adapt the proof of Theorem \ref{point conditional density e} to the present setting it is necessary to get some insight on the plausible values of $Y_{1}$ under $P_{a_{n}}.$ It holds

\begin{lemma}
\label{LemmaOrderof Y}It holds
\[
Y_{1}=O_{P_{a_{n}}}\left(a_{n}\right).
\]
\end{lemma}

\begin{proof}
This is a consequence of Markov Inequality:
\[
P\left(\left.Y_{1}>u\right\vert S_{1}^{n}=na_{n}\right)\leq\frac{E\left(\left.Y_{1}\right\vert S_{1}^{n}=na_{n}\right)}{u}=\frac{a_{n}}{u}
\]
\bigskip which goes to $0$ for all $u=u_{n}$ such that $\lim_{n\rightarrow\infty}u_{n}/a_{n}=\infty$. 

Now making use of Lemma \ref{LemmaOrderof Y} in the proof of Theorem \ref{pThm oint conditional density sous o(1)} and Theorem \ref{Thm point condit density sous O(1)}, substituting $y_{1}$ with $Y_{1},$ completes the proof.
\end{proof}

Denote the probability measures $P_{a_{n}}$ , $\Pi^{a_{n}}$ and $G_{a_{n}}$ with respective densities $p_{a_{n}},\pi^{a_{n}}$ and $g_{a_{n}}$ .

\subsection{\label{SubsVariationNorm}Gibbs principle in variation norm}

We now consider the approximation of $P_{a_{n}}$ by $G_{a_{n}}$ in variation norm.

The main ingredient is the fact that in the present setting approximation of  $p_{a_{n}}$ by $g_{a_{n}}$ in probability plus some rate implies approximation of the corresponding measures in variation norm. This approach has been developped in \cite{Bronia}; we state a first lemma which states that wether two densities are equivalent in probability with small relative error when measured according to the first one, then the same holds under the sampling of the second.

Let $\mathfrak{R}_{n}$ and $\mathfrak{S}_{n}$ denote two p.m's on $\mathbb{R}^{n}$ with respective densities $\mathfrak{r}_{n}$ and $\mathfrak{s}_{n}.$

\begin{lemma}
\label{Lemma:commute_from_p_n_to_g_n} Suppose that for some sequence $\varpi_{n}$ which tends to $0$ as $n$ tends to infinity
\begin{equation}
\mathfrak{r}_{n}\left(Y_{1}^{n}\right)=\mathfrak{s}_{n}\left(Y_{1}^{n}\right)\left(1+o_{\mathfrak{R}_{n}}(1)\right)
\label{p_n equiv g_n under p_n}
\end{equation}
as $n$ tends to $\infty$. Then 
\begin{equation}
\mathfrak{s}_{n}\left(Y_{1}^{n}\right)=\mathfrak{r}_{n}\left(Y_{1}^{n}\right)\left(1+o_{\mathfrak{S}_{n}}(1)\right).
\label{g_n equiv p_n under g_n}
\end{equation}
\end{lemma}

The proof of this result is available in \cite{Bronia}. Applying this Lemma to the present setting yields
\[
g_{a_{n}}\left( Y_{1}\right) =p_{a_{n}}\left( Y_{1}\right) \left(
1+o_{G_{a_{n}}}\left( 1/\sqrt{n}\right) \right)
\]
as $n\rightarrow\infty,$ which together with Theorem \ref{Thm point condit density sous O(1)} or Theorem \ref{pThm oint conditional density sous o(1)} implies 
\[
p_{a_{n}}\left(Y_{1}\right)=g_{a_{n}}\left(Y_{1}\right)\left(1+o_{P_{a_{n}}}\left(1/\sqrt{n}\right)\right)
\]
or
\[
p_{a_{n}}\left(Y_{1}\right)=\pi^{a_{n}}\left(Y_{1}\right)\left(1+o_{P_{a_{n}}}\left(1/\sqrt{n}\right)\right)
\]

This fact entails

\begin{theorem}
\label{ThmcvVarTot}Under (\ref{croissance de aen O(1)}) the total variation norm between $P_{a_{n}}$ and $G_{a_{n}}$ goes to $0$ as $n\rightarrow \infty$. When (\ref{croissance de aen o(1)}) holds then the total variation norm between $P_{a_{n}}$ and $\Pi^{a_{n}}$ goes to $0$ as $n\rightarrow\infty$.
\end{theorem}

The proof of this theorem is also provided in \cite{Bronia}.

\begin{remark}
This result is to be paralleled with Theorem 1.6 in Diaconis and Freedman \cite{Diaconis1} and Theorem 2.15 in Dembo and Zeitouni \cite{Dembo} which provide a rate for this convergence in the LDP range.
\end{remark}

\subsection{The asymptotic location of $X$ under the conditioned distribution}

\label{SectAsymptNormalTilted}This paragraph intends to provide some insight on the behaviour of $X_{1}$ under the condition $\left(S_{1}^{n}=na_{n}\right)$; this will be extended further on to the case when $\left(S_{1}^{n}\geq na_{n}\right)$ and to be considered in parallel with similar facts developped in \cite{Bronia} for larger values of $a_{n}$.

Let $\mathcal{X}_{t}$ be a r.v. with density $\pi^{a_{n}}$ where $m(t)=a_{n}$ and $a_{n}$ satisfies (\ref{croissance de aen o(1)}) or (\ref{croissance de aen O(1)}). Recall that $E\mathcal{X}_{t}=a_{n}$ and $Var\mathcal{X}_{t}=s^{2}$. We evaluate the moment generating function of the normalized variable $\left(\mathcal{X}_{t}-a_{n}\right)/s$. It holds
\[
\log E\lbrack\exp(\lambda\left(\mathcal{X}_{t}-a_{n}\right)/s)\rbrack=-\lambda a_{n}/s+\log\Phi\left(t+\frac{\lambda}{s}\right)-\log\Phi\left(t\right).
\]
A second order Taylor expansion in the above display yields
\[
\log E\lbrack\exp(\lambda\left(\mathcal{X}_{t}-a_{n}\right)/s)\rbrack=\frac{\lambda^{2}}{2}\frac{s^{2}\left(t+\frac{\theta\lambda}{s}\right)}{s^{2}}
\]
where $\theta=\theta(t,\lambda)\in\left(0,1\right)$. The proof of the following Lemma is deferred to the Section \ref{appendix}. It holds

\begin{lemma}
\label{Lemme s^2 self neglecting}Under the above hypotheses and notation, for any compact set $K,$
\[
\lim_{n\rightarrow\infty}\sup_{u\in K}\frac{s^{2}\left(t+\frac{u}{s}\right)}{s^{2}}=1.
\]
\end{lemma}

Applying the above Lemma it follows that the normalized r.v's $\left(\mathcal{X}_{t}-a_{n}\right)/s$ converge to a standard normal variable $N(0,1)$ in distribution, as $n\rightarrow\infty$. This amounts to say that 
\[
\mathcal{X}_{t}=a_{n}+sN(0,1)+o_{\Pi^{a_{n}}}(1).
\]
which implies that $\mathcal{X}_{t}$ concentrates around $a_{n}$ with rate $s$. Due to Theorem \ref{ThmcvVarTot} the same holds for $X_{1}$ under $\left(S_{1}^{n}=na_{n}\right)$.

\subsection{Conditional limit behaviour under other mean effect events}

Let $X_{1},\ldots,X_{n}$ denote $n$ i.i.d. real valued r.v's with distribution $P $ and density $p$ and let $f:\mathbb{R\rightarrow R}$ be a measurable function such that $\Phi_{f}(\lambda):=E\lbrack\exp(\lambda f(X_{1}))\rbrack$ is finite for $\lambda$ in a non void neighborhood of $0$ (the so-called Cramer condition). Denote $m_{f}(\lambda)$ and $s_{f}^{2}(\lambda)$ the first and second derivatives of $\log\Phi_{f}(\lambda)$. Assume that the r.v. $f(X_{1})$ has density $p_{f}$ on $\mathbb{R}$, and denote $p_{f}(f(X_{1})=u)$ its value at point $u.$

Denote 
\[
\pi_{f}^{a}(y)=\frac{\exp\lambda y}{\Phi_{f}(\lambda)}p_{f}(y)
\]
with $\lambda $ the unique solution of the equation $m_{f}(\lambda )=a$ for all $a$ in $Im(f(X_{1}))$ assuming that $\lambda\rightarrow\Phi_{f}(\lambda)$ is steep on its domain. Denote 
\[
F_{i}^{j}:=f(X_{i})+\ldots+f(X_{j})
\]
for $1\leq i\leq j\leq n$. We make use of the following equality
\begin{align*}
p(\left.X_{1}=x\right\vert F_{1}^{n}& =na_{n}) \\
& =\frac{p(X_{1}=x)}{p_{f}(f(X_{1})=f(x))}\times \\
& \left(p_{f}(f(X_{1})=f(x))\frac{p_{f}(\left. f(X_{1})=f(x)\right\vert F_{1}^{n}=na_{n})}{p_{f}(\left.f(X_{1})=f(x)\right\vert F_{2}^{n}=na_{n}-f(x))}\right).
\end{align*}
Note that for all $\alpha$ in $Im(f(X_{1}))$, denoting $\lambda$ the solution of $m_{f}(\lambda)=\alpha$ and defining 
\[
_{f}\pi^{\alpha}(x):=\frac{e^{\lambda f(x)}p(X=x)}{\int e^{\lambda f(x)}p(X=x)dx}
\]
it is readily checked that

\[
_{f}\pi^{\alpha}(x)=\frac{p(X_{1}=x)}{p_{f}(f(X_{1})=f(x))}\pi_{f}^{a_{n}}(f(x)).
\]
Denoting $P_{a_{n},f}$ the distribution of $X_{1}$ given $\left(F_{1}^{n}=na_{n}\right)$ it results, using Theorem \ref{ThmLocalProba} that the following Theorem holds.

\begin{theorem}
\label{ThmGibbsMoyenneGenerale}Assume that, with $s$ substituted by $s_{f}$, condition (\ref{croissance de aen o(1)}) holds. Then 
\[
p(\left.X_{1}=Z\right\vert F_{1}^{n}=na_{n})=_{f}\pi^{\alpha}(Z)\left(1+o_{P_{a_{n},f}}\left(1/\sqrt{n}\right)\right) 
\]
and under (\ref{croissance de aen O(1)})
\begin{eqnarray*}
p(\left.X_{1}=Z\right\vert F_{1}^{n} &=&na_{n}) \\
&=&C\frac{p(X_{1}=Z)}{p_{f}(f(X_{1})=f(Z))}\mathfrak{n}\left(\alpha\beta+a_{n},\beta,f(Z)\right)(1+o_{P_{a_{n},f}}\left(1/\sqrt{n}\right))
\end{eqnarray*}
where $\alpha :=\alpha_{n}$ and $\beta:=\beta_{n}$ are defined in (\ref{alfa_i}) and (\ref{beta_i}) with $m$ and $s$ substituted by $m_{f}$ and $s_{f}$.
\end{theorem}

\begin{remark}
The first part of the above Theorem extends the classical Gibbs Principle under condition (\ref{croissance de aen o(1)}), which, for fixed $a=a_{n}$ writes 
\[
p(\left.X_{1}=x\right\vert F_{1}^{n}=na)=_{f}\pi^{\alpha}(x)\left(1+o\left(1/\sqrt{n}\right)\right)
\]
for any fixed $x$. See \cite{Diaconis1}. This statement does not hold any
longer under condition (\ref{croissance de aen O(1)}).
\end{remark}

\begin{remark}
Making use of the same arguments as in Subsection \ref{SubsVariationNorm} it follows that Theorem \ref{ThmGibbsMoyenneGenerale} yields that the variation distance between the conditional distribution and its approximation tends to $0$ as $n$ tends to infinity.
\end{remark}

\begin{example}
Consider for example the application of the above result to r.v's $Y_{1},\ldots,Y_{n}$ with $Y_{i}:=\left(X_{i}\right)^{2}$ where the $X_{i}^{\prime}s$ are i.i.d. and are such that the density of the i.i.d. r.v's $Y_{i}^{\prime}s$ satisfy (\ref{2.1}), where $h\in R_{\beta}\cup R_{\infty}$ with $\beta>1$. By the Gibbs conditional principle, for \textit{fixed} $a$, conditionally on $\left(\sum_{i=1}^{n}Y_{i}=na\right)$ the generic r.v. $Y_{1}$ has a non degenerate limit distribution 
\[
p_{Y}^{\ast}(y):=\frac{\exp ty}{E\exp tY_{1}}p_{Y}(y)
\]
and the limit density of $X_{1}$ under $\left(\sum_{i=1}^{n}X_{i}^{2}=na\right)$ is

\[
p_{X}^{\ast}(y):=\frac{\exp tx^{2}}{E\exp tX_{1}^{2}}p_{X}(y)
\]
whereas, when $a_{n}\rightarrow\infty$, $Y_{1}$'s the limit conditional distribution is degenerate and concentrates around $a_{n}.$ As a consequence the distribution of $X_{1}$ under the condition $\left(\sum_{i=1}^{n}X_{i}^{2}=na_{n}\right)$ concentrates sharply at $-\sqrt{a_{n}}$ and $+\sqrt{a_{n}}$.
\end{example}

\section{EDP under exceedance}\label{secExceedance}

The following proposition states the marginally conditional density under condition $A_{n}=\{S_{1}^{n}\geq na_{n}\}$. We denote this density by $p_{A_{n}}$ to differentiate it from $p_{a_{n}}$ which is under condition $\{S_{1}^{n}=na_{n}\}$. For the purpose of the proof, we need the following Lemma, based on Theorem $6.2.1$ of Jensen \cite{Jensen} in order to provide the asymptotic estimation of the tail probability $P(S_{1}^{n}\geq na_{n})$ and of the $n$-convolution density $p(S_{1}^{n}/n=u)$ for $u>a_{n}$.

Define 
\[
I(x):=xm^{-1}(x)-\log\Phi\left( m^{-1}(x)\right).
\]

We make use of the following result (see Section \ref{appendix} for the proof).

\begin{lemma}
\label{JensenLemme}Set $m(t)=a_{n}$. Suppose that $a_{n}\rightarrow\infty$ as $n\rightarrow\infty.$ Then it holds 
\begin{equation}
P(S_{1}^{n}\geq na_{n})=\frac{\exp(-nI(a_{n}))}{\sqrt{2\pi}\sqrt{n}ts(t)}\left(1+o\left(\frac{1}{\sqrt{n}}\right)\right).  \label{Hoglund}
\end{equation}
Let further $t_{\tau}$ be defined by $m(t_{\tau})=\tau$ with $\tau\geq a_{n}$, it then holds, uniformly upon $\tau$ 
\begin{equation}
p(S_{1}^{n}=n\tau)=\frac{\sqrt{n}\exp(-nI(\tau))}{\sqrt{2\pi}s(t_{\tau})}\left(1+o\left(\frac{1}{\sqrt{n}}\right)\right).
\label{Hoglund density}
\end{equation}
\end{lemma}

The proof of this lemma is postponed to the Section \ref{appendix}.

\begin{theorem}
\label{main theorem}Let $X_{1},\ldots,X_{n}$ be i.i.d. random variables with density $p(x)$ defined in (\ref{2.1}) and $h(x)\in\mathcal{R}$. Set $m(t)=a_{n}$ let $\eta_{n}$ be a positive sequence satisfying 
\[
\eta_{n}\longrightarrow 0\qquad\emph{and}\qquad nm^{-1}(a_{n})\eta_{n}\longrightarrow \infty.
\]
(i) When (\ref{croissance de aen o(1)}) holds 
\[
p_{A_{n}}(y_{1})=p(X_{1}=y_{1}|S_{1}^{n}\geq na_{n})=\pi_{A_{n}}(y_{1})\left(1+o\left(\frac{1}{\sqrt{n}}\right)\right),
\]
with 
\[
\pi_{A_{n}}(y_{1})=ts(t)e^{nI(a_{n})}\int_{a_{n}}^{a_{n}+\eta_{n}}\pi_{\tau }(y_{1})\exp \left(-nI(\tau)-\log s(t_{\tau})\right)d\tau
\]
with $t_{\tau}$ defined by $m(t_{\tau})=\tau$.

(ii) When (\ref{croissance de aen O(1)}) holds 
\[
p_{A_{n}}(y_{1})=p(X_{1}=y_{1}|S_{1}^{n}\geq na_{n})=g_{A_{n}}(y_{1})\Big(1+o\big(\frac{1}{\sqrt{n}}\big)\Big),
\]
with 
\[
g_{A_{n}}(y_{1})=ts(t)e^{nI(a_{n})}\int_{a_{n}}^{a_{n}+\eta_{n}}g_{\tau}(y_{1})\exp \left(-nI(\tau)-\log s(t_{\tau})\right)d\tau,
\]
where $g_{\tau}=\pi^{\tau}$ with $t_{\tau}$ defined by $m(t_{\tau})=\tau$.
\end{theorem}

The proof is postponed to the Section \ref{appendix}.

\begin{remark}
Conditions (\ref{croissance de aen O(1)}) and (\ref{croissance de aen o(1)}) have to be compared with the growth condition pertaining to the sequence $a_{n}$ for which (\ref{democracy}) holds. Consider the regularly varying case, namely assume that $h(x)=x^{\beta}l(x)$ for some $\beta>0$. Then making use of Theorem \ref{THM3.1} it is readily checked that (\ref{croissance de aen O(1)}) amounts to 
\begin{equation}
\lim\inf_{n\rightarrow\infty}\frac{a_{n}}{n^{1/\left(1+\beta\right)}}>0.  \label{croissrapide}
\end{equation}
Now (\ref{democracy}) holds whether for some $\delta >1/\left(\beta+1\right)$
\begin{equation}
\lim \inf_{n\rightarrow\infty}\frac{a_{n}}{n^{\delta}}>0.  \label{demo}
\end{equation}
Assume that (\ref{demo}) holds; then (\ref{croissrapide}) holds for all distributions $p$ with $h(x)=x^{\beta}l(x)$ and $\beta>\left(1-\delta\right)/\delta$. This can be stated as follows: Assume that for some $0<\eta<1$ it holds 
\[
\lim\inf_{n\rightarrow\infty}\frac{a_{n}}{n^{\eta}}>0
\]
then whenever $\beta>\left(1-\eta \right)/\eta$ (\ref{demo}) and (\ref{croissrapide}) simultaneously hold.
\end{remark}

\section{Appendix}\label{appendix}

\subsection{Proof of Theorem \protect\ref{3theorem1}}

We state a preliminary Lemma, whose role is to provide some information on the characteristic function of the normalised random variable $\left(\mathcal{X}_{t}-m(t)\right)/s(t)$ with density $\widetilde{\pi}_{t}$ defined by 
\begin{equation}
\widetilde{\pi}_{t}(x):=\frac{s(t)\exp t\left(s(t)x+m(t)\right)p(s(t)x+m(t))}{\phi(t)}  \label{p_t}
\end{equation}
as $t\rightarrow\infty$. The density $p$ satisfies the hypotheses in Section \ref{notation}. Denote $\varphi^{a_{n}}(u):=\int e^{iux}\widetilde{\pi}_{t}(x)dx$ the characteristic function of $\left(\mathcal{X}_{t}-m(t)\right)/s(t)$. It holds

\begin{lemma}
\label{LemmaBorne}Assume that there exists $c_{1},c_{2}$ both positive such that for all $t$
\begin{equation}
\widetilde{\pi}_{t}(x)>c_{1}\text{ for }\left\vert x\right\vert<c_{2}
\label{regul}
\end{equation}
then under the hypotheses stated in Section \ref{notation}, for any $c>0$ there exists $\rho<1$ such that 
\begin{equation}
\left\vert\varphi^{a_{n}}(u)\right\vert\leq\rho
\label{borne-f.c.}
\end{equation}
for $\left\vert u\right\vert>c$ and all $a_{n}$.
\end{lemma}

\begin{proof} The proof of this Lemma is in \cite{Jensen}, p150; we state it for completeness. Assume (\ref{regul}) holds with $\tilde{\pi}_{t}(x)>c_{1}$ for $\left\vert x\right\vert>c_{2}$ and setting $\epsilon:=c_{2}/2$
\begin{eqnarray*}
\left\vert\varphi^{a_{n}}(u)\right\vert &\leq &\left\vert\int e^{izu}{\Large 1}\left( \left\vert z\right\vert<\epsilon\right)c_{1}dz\right\vert+\int\left\{\tilde{\pi}_{t}(z)-{\Large 1}\left(\left\vert z\right\vert<\epsilon\right)c_{1}\right\}dz \\
&\leq &c_{1}\left(2\epsilon\right)\left\vert\frac{e^{iu\epsilon}-e^{-iu\epsilon}}{2iu\epsilon}\right\vert+\left\{1-2\epsilon c_{1}\right\}
\end{eqnarray*}
and the last expression is independent on $a_{n}$ and is such that for any $c>0$ there exists $\rho<1$ such that the expression is less than $\rho$ for $\left\vert u\right\vert>c.$

For the density function $p(x)$ Theorem 5.4 of Nagaev \cite{Nagaev} states that the normalized tilted density of $p(x)$, namely, $\widetilde{\pi}_{t}(x)$ has the property 
\begin{equation}
\lim_{a_{n}\rightarrow\infty}\sup_{x\in\mathbb{R}}|\widetilde{\pi}_{t}(x)-\varphi(x)|=0  \label{supTiltedGauss}
\end{equation}
which proves (\ref{regul}).
\end{proof}

We now turn to the Proof of Theorem \ref{3theorem1}.

Since the proof is based on characteristic function (c.f.) arguments, we will use the following notation, in accordance with the common use in this area, therefore turning from laplace transform notation to characteristic function ones. Recall that we denote $\widetilde{\pi}_{t}$ the normalized conjugate density of $p(x)$. Also $\rho_{n}$ is the normalized $n-$fold convolution of $\widetilde{\pi}_{t}.$ Hence we consider the triangular array whose $n-$th row consists in $n$ i.i.d. copies of a r.v. with standardized density $\widetilde{\pi}_{t}$ and the sum of the row, divided by $\sqrt{n}$, has density $\rho_{n}.$ The standard Gaussian density is denoted $\phi$. The c.f. of $\widetilde{\pi}_{t}$ is denoted $\varphi^{a_{n}}$ so that the c.f. of $\rho_{n}$ is $\left(\varphi^{a_{n}}(.)\right)^{n},$ and $m(t)=a_{n}$.

\textbf{Step 1:} In this step, we will express the following formula $G(x)$ by its Fourier transform. Let 
\[
G(x):=\rho_{n}(x)-\phi(x)-\frac{\mu_{3}}{6\sqrt{n}s_{n}^{3}}\left(x^{3}-3x\right)\phi (x).
\]
From 
\[
\phi(x):=\frac{1}{2\pi}\int_{-\infty}^{\infty}e^{-i\tau x}e^{-\frac{1}{2}\tau^{2}}d\tau,
\]
it follows that 
\[
\phi^{\prime\prime\prime}(x)=-\frac{1}{2\pi}\int_{-\infty}^{\infty}(i\tau)^{3}e^{-i\tau x}e^{-\frac{1}{2}\tau^{2}}d\tau.
\]
On the other hand 
\[
\phi^{\prime\prime\prime}(x)=-(x^{3}-3x)\phi(x),
\]
which gives 
\begin{equation}
(x^{3}-3x)\phi(x)=\frac{1}{2\pi}\int_{-\infty}^{\infty}(i\tau)^{3}e^{-i\tau x}e^{-\frac{1}{2}\tau^{2}}d\tau.  \label{3the11}
\end{equation}
By Fourier inversion 
\begin{equation}
\rho_{n}(x)=\frac{1}{2\pi}\int_{-\infty}^{\infty}e^{-i\tau x}\left(\varphi^{a_{n}}(\tau/\sqrt{n})\right)^{n}d\tau.  \label{3the12}
\end{equation}
Using (\ref{3the11}) and (\ref{3the12}), we have 
\[
G(x)=\frac{1}{2\pi}\int_{-\infty}^{\infty}e^{-i\tau x}\left(\varphi^{a_{n}}(\tau/\sqrt{n})^{n}-e^{-\frac{1}{2}\tau^{2}}-\frac{\mu_{3}}{6\sqrt{n}s^{3}}(i\tau)^{3}e^{-\frac{1}{2}\tau^{2}}\right)d\tau.
\]
Hence it holds 
\begin{align*}
& \Big|\rho_{n}(x)-\phi(x)-\frac{\mu_{3}}{6\sqrt{n}s^{3}}\left(x^{3}-3x\right)\phi(x)\Big| \\
& \leq\frac{1}{2\pi}\int_{-\infty}^{\infty}\left|\left(\varphi^{a_{n}}(\tau/\sqrt{n})\right)^{n}-e^{-\frac{1}{2}\tau^{2}}-\frac{\mu_{3}}{6\sqrt{n}s^{3}}(i\tau)^{3}e^{-\frac{1}{2}\tau^{2}}\right|d\tau.
\end{align*}
\textbf{Step 2:} In this step, we show that for large $n$, the characteristic function $\varphi^{a_{n}}$ satisfies 
\[
\int|\varphi^{a_{n}}(\tau)|^{2}d\tau<\infty
\]
By Parseval identity 
\[
\int|\varphi^{a_{n}}(\tau)|^{2}d\tau=2\pi\int(\widetilde{\pi}_{t}(x))^{2}dx\leq 2\pi\sup_{x\in\mathbb{R}}\widetilde{\pi}_{t}(x)<\infty.
\]
Use (\ref{supTiltedGauss}) to conclude the proof.

\textbf{Step 3:} In this step, we complete the proof by showing that when $n\rightarrow \infty$ 
\begin{equation}
\int_{-\infty}^{\infty}\Big|\big(\varphi^{a_{n}}(\tau/\sqrt{n})\big)^{n}-e^{-\frac{1}{2}\tau^{2}}-\frac{\mu_{3}}{6\sqrt{n}s^{3}}(i\tau)^{3}e^{-\frac{1}{2}\tau^{2}}\Big|d\tau=o\Big(\frac{1}{\sqrt{n}}\Big).\label{3theo100}
\end{equation}

The LHS in (\ref{3theo100}) is splitted on $|\tau|>\omega\sqrt{n}$ and on $|\tau|\leq\omega\sqrt{n}$. It holds 
\begin{align}
& \sqrt{n}\int_{|\tau|>\omega\sqrt{n}}\Big|\big(\varphi^{a_{n}}(\tau/\sqrt{n})\big)^{n}-e^{-\frac{1}{2}\tau^{2}}-\frac{\mu_{3}}{6\sqrt{n}s^{3}}(i\tau )^{3}e^{-\frac{1}{2}\tau^{2}}\Big|d\tau   \notag \\
& \leq\sqrt{n}\int_{|\tau|>\omega\sqrt{n}}\Big|\big(\varphi^{a_{n}}(\tau/\sqrt{n})\big)\Big|^{n}d\tau+\sqrt{n}\int_{|\tau|>\omega\sqrt{n}}\Big|e^{-\frac{1}{2}\tau^{2}}+\frac{\mu_{3}}{6\sqrt{n}s^{3}}(i\tau)^{3}e^{-\frac{1}{2}\tau^{2}}\Big|d\tau   \notag \\
& \leq\sqrt{n}\rho^{n-2}\int_{|\tau|>\omega\sqrt{n}}\Big|\big(\varphi^{a_{n}}(\tau/\sqrt{n})\big)\Big|^{2}d\tau+\sqrt{n}\int_{|\tau|>\omega\sqrt{n}}e^{-\frac{1}{2}\tau ^{2}}\Big(1+\Big|\frac{\mu_{3}\tau^{3}}{6\sqrt{n}s^{3}}\Big|\Big)d\tau.  \label{line3}
\end{align}
where we used Lemma \ref{LemmaBorne} from the second line to the third one. The first term of the last line tends to $0$ when $n\rightarrow\infty$, since 
\begin{align*}
& \sqrt{n}\rho^{n-2}\int_{|\tau|>\omega\sqrt{n}}\Big|\big(\varphi^{a_{n}}(\tau/\sqrt{n})\big)\Big|^{2}d\tau  \\
& =\exp\left(\frac{1}{2}\log n+(n-2)\log\rho+\log\int_{|\tau|>\omega\sqrt{n}}\left(\varphi^{a_{n}}(\tau/\sqrt{n})\right)^{2}d\tau\right)\longrightarrow 0.
\end{align*}
By Corollary \ref{3cor1} when $n\rightarrow\infty$ 
\begin{align*}
& \sqrt{n}\int_{|\tau|>\omega\sqrt{n}}e^{-\frac{1}{2}\tau^{2}}\Big(1+\Big|\frac{\mu_{3}\tau^{3}}{6\sqrt{n}s^{3}}\Big|\Big)d\tau  \\
& \leq\sqrt{n}\int_{|\tau|>\omega\sqrt{n}}e^{-\frac{1}{2}\tau^{2}}|\tau|^{3}d\tau =\sqrt{n}\int_{|\tau|>\omega\sqrt{n}}\exp\Big\{-\frac{1}{2}\tau^{2}+3\log|\tau|\Big\}d\tau  \\
& =2\sqrt{n}\exp\big(-\omega^{2}n/2+o(\omega^{2}n/2)\big)\longrightarrow 0,
\end{align*}
where the second equality holds from, for example, Chapiter $4$ of \cite{Bingham}. Summing up, when $n\rightarrow\infty$ 
\[
\int_{|\tau|>\omega\sqrt{n}}\Big|\big(\varphi^{a_{n}}(\tau/\sqrt{n})\big)^{n}-e^{-\frac{1}{2}\tau^{2}}-\frac{\mu_{3}}{6\sqrt{n}s^{3}}(i\tau)^{3}e^{-\frac{1}{2}\tau^{2}}\Big|d\tau=o\Big(\frac{1}{\sqrt{n}}\Big).
\]
If $|\tau|\leq\omega\sqrt{n}$, it holds 
\begin{align}
& \int_{|\tau|\leq\omega\sqrt{n}}\Big|\big(\varphi^{a_{n}}(\tau/\sqrt{n})\big)^{n}-e^{-\frac{1}{2}\tau^{2}}-\frac{\mu_{3}}{6\sqrt{n}s^{3}}(i\tau)^{3}e^{-\frac{1}{2}\tau^{2}}\Big|d\tau   \notag \\
& =\int_{|\tau|\leq\omega\sqrt{n}}e^{-\frac{1}{2}\tau^{2}}\Big|\big(\varphi^{a_{n}}(\tau/\sqrt{n})\big)^{n}e^{\frac{1}{2}\tau^{2}}-1-\frac{\mu_{3}}{6\sqrt{n}s^{3}}(i\tau)^{3}\Big|d\tau   \notag \\
& =\int_{|\tau|\leq\omega\sqrt{n}}e^{-\frac{1}{2}\tau^{2}}\Big|\exp\Big\{n\log\varphi^{a_{n}}(\tau/\sqrt{n})+{\frac{1}{2}\tau^{2}}\Big\}-1-\frac{\mu_{3}}{6\sqrt{n}s^{3}}(i\tau)^{3}\Big|d\tau.  \label{line3bis}
\end{align}
The integrand in the last display is bounded through 
\[
|e^{\alpha}-1-\beta|=|(e^{\alpha}-e^{\beta})+(e^{\beta}-1-\beta)|\leq(|\alpha-\beta|+\frac{1}{2}\beta^{2})e^{\gamma},
\]
where $\gamma\geq\max(|\alpha|,|\beta|)$; this inequality follows replacing $e^{\alpha},e^{\beta}$ by their power series, for real or complex $\alpha,\beta$. Denote by 
\[
\gamma(\tau)=\log\varphi^{a_{n}}(\tau)+{\frac{1}{2}\tau^{2}}.
\]
Since $\gamma^{\prime}(0)=\gamma^{\prime\prime}(0)=0$, the third order Taylor expansion of $\gamma(\tau)$ at $\tau=0$ yields 
\[
\gamma(\tau)=\gamma(0)+\gamma^{\prime}(0)\tau+\frac{1}{2}\gamma^{\prime\prime}(0)\tau^{2}+\frac{1}{6}\gamma^{\prime\prime\prime}(\xi)\tau^{3}=\frac{1}{6}\gamma^{\prime\prime\prime}(\xi)\tau^{3},
\]
where $0<\xi<\tau$. Hence it holds 
\[
\Big|\gamma(\tau)-\frac{\mu_{3}}{6s^{3}}(i\tau)^{3}\Big|=\Big|\gamma^{\prime\prime\prime}(\xi)-\frac{\mu_{3}}{s_{n}^{3}}i^{3}\Big|\frac{\tau^{3}}{6}.
\]
Here $\gamma^{\prime\prime\prime}$ is continuous; thus we can choose $\omega$ small enough such that $|\gamma^{\prime\prime\prime}(\xi)|<\rho$ for $|\tau|<\omega$. Meanwhile, for $n$ large enough, according to Corollary \ref{3cor1}, we have $\mu_{3}/s^{3}\rightarrow 0$. Hence it holds for $n$ large enough 
\begin{equation}
\Big|\gamma(\tau)-\frac{\mu_{3}}{6s^{3}}(i\tau)^{3}\Big|\leq\Big(|\gamma^{\prime\prime\prime}(\xi)|+\rho\Big)\frac{|\tau|^{3}}{6}<\rho\tau^{3}.  \label{ineg}
\end{equation}
Choose $\omega$ small enough, such that for $n$ large enough it holds for $|\tau|<\omega$ 
\[
\Big|\frac{\mu_{3}}{6s^{3}}(i\tau)^{3}\Big|\leq\frac{1}{4}\tau^{2},\text{ and }\gamma(\tau)|\leq\frac{1}{4}\tau^{2}.
\]
For this choice of $\omega$, when $|\tau|<\omega$ we have 
\[
\max\Big(\Big|\frac{\mu_{3}}{6s^{3}}(i\tau)^{3}\Big|,|\gamma(\tau)|\Big)\leq\frac{1}{4}\tau^{2}.
\]
Replacing $\tau$ by $\tau/\sqrt{n}$, it holds for $|\tau|<\omega\sqrt{n}$, and using (\ref{ineg}) 
\begin{align*}
& \Big|n\log\varphi^{a_{n}}(\tau/\sqrt{n})+{\frac{1}{2}\tau^{2}}-\frac{\mu_{3}}{6\sqrt{n}s^{3}}(i\tau)^{3}\Big| \\
& =n\Big|\log\varphi^{a_{n}}(\tau/\sqrt{n})+{\frac{1}{2}\Big(\frac{\tau}{\sqrt{n}}\Big)^{2}}-\frac{\mu_{3}}{6s^{3}}\Big(\frac{i\tau}{\sqrt{n}}\Big)^{3}\Big| \\
& =n\Big|\gamma\Big(\frac{\tau}{\sqrt{n}}\Big)-\frac{\mu_{3}}{6s^{3}}\Big(\frac{i\tau}{\sqrt{n}}\Big)^{3}\Big|<\frac{\rho|\tau|^{3}}{\sqrt{n}}.
\end{align*}
In a similar way, it also holds for $|\tau|<\omega\sqrt{n}$ 
\begin{align*}
& \max\Big(\Big|n\log\varphi^{a_{n}}(\tau/\sqrt{n})+{\frac{1}{2}\tau^{2}}\Big|,\Big|\frac{\mu_{3}}{6\sqrt{n}s^{3}}(i\tau)^{3}\Big|\Big) \\
& =n\max\Big(\Big|\gamma\Big(\frac{\tau}{\sqrt{n}}\Big)\Big|,\Big|\frac{\mu_{3}}{6s^{3}}\Big(\frac{i\tau}{\sqrt{n}}\Big)^{3}\Big|\Big)\leq\frac{1}{4}\tau^{2}.
\end{align*}
Turn to the integrand in (\ref{line3bis}). We then for $|\tau|<\omega\sqrt{n}$ 
\begin{align*}
& \Big|\exp\Big\{n\log\varphi^{a_{n}}(\tau/\sqrt{n})+{\frac{1}{2}\tau^{2}}\Big\}-1-\frac{\mu_{3}}{6\sqrt{n}s^{3}}(i\tau)^{3}\Big| \\
& \leq\Big(\Big|n\log\varphi^{a_{n}}(\tau/\sqrt{n})+{\frac{1}{2}\tau^{2}}-\frac{\mu_{3}}{6\sqrt{n}s^{3}}(i\tau)^{3}\Big|+\frac{1}{2}\Big|\frac{\mu_{3}}{6\sqrt{n}s^{3}}(i\tau)^{3}\Big|^{2}\Big) \\
& \qquad\times\exp\Big[\max\Big(\Big|n\log\varphi^{a_{n}}(\tau/\sqrt{n})+{\frac{1}{2}\tau^{2}}\Big|,\Big|\frac{\mu_{3}}{6\sqrt{n}s^{3}}(i\tau)^{3}\Big|\Big)\Big] \\
& \leq\Big(\frac{\rho|\tau|^{3}}{\sqrt{n}}+\frac{1}{2}\Big|\frac{\mu_{3}}{6\sqrt{n}s^{3}}(i\tau)^{3}\Big|^{2}\Big)\exp\Big(\frac{\tau^{2}}{4}\Big)\\
& =\Big(\frac{\rho|\tau|^{3}}{\sqrt{n}}+\frac{\mu_{3}^{2}\tau^{6}}{72ns^{6}}\Big)\exp \Big(\frac{\tau^{2}}{4}\Big).
\end{align*}
Use this upper bound to obtain 
\begin{align*}
& \int_{|\tau|\leq\omega\sqrt{n}}\Big|\big(\varphi^{a_{n}}(\tau/\sqrt{n})\big)^{n}-e^{-\frac{1}{2}\tau ^{2}}-\frac{\mu_{3}}{6\sqrt{n}s^{3}}(i\tau)^{3}e^{-\frac{1}{2}\tau^{2}}\Big|d\tau  \\
& \leq\int_{|\tau|\leq\omega\sqrt{n}}\exp\Big(-\frac{\tau^{2}}{4}\Big)\Big(\frac{\rho|\tau|^{3}}{\sqrt{n}}+\frac{\mu_{3}^{2}\tau^{6}}{72ns^{6}}\Big)d\tau  \\
& =\frac{\rho}{\sqrt{n}}\int_{|\tau|\leq\omega\sqrt{n}}\exp\Big(-\frac{\tau^{2}}{4}\Big)|\tau|^{3}d\tau+\frac{\mu_{3}^{2}}{72ns^{6}}\int_{|\tau|\leq\omega\sqrt{n}}\exp\Big(-\frac{\tau^{2}}{4}\Big)\tau^{6}d\tau,
\end{align*}
where both the first integral and the second integral are finite, and $\rho$ is arbitrarily small; use Corollary \ref{3cor1}, to obtain 
\[
\int_{|\tau|\leq\omega\sqrt{n}}\Big|\big(\varphi^{a_{n}}(\tau/\sqrt{n})\big)^{n}-e^{-\frac{1}{2}\tau^{2}}-\frac{\mu_{3}}{6\sqrt{n}s^{3}}(i\tau)^{3}e^{-\frac{1}{2}\tau^{2}}\Big|d\tau=o\Big(\frac{1}{\sqrt{n}}\Big).
\]
This gives (\ref{3theo100}), and therefore we obtain 
\[
\Big|\bar{\pi}_{n}^{a_{n}}(x)-\phi(x)-\frac{\mu_{3}}{6\sqrt{n}s^{3}}\big(x^{3}-3x\big)\phi(x)\Big|=o\Big(\frac{1}{\sqrt{n}}\Big),
\]
which concludes the proof.

\subsection{Proof of Theorem \ref{pThm oint conditional density sous o(1)}}

It is well known and easily checked that the conditional density $p(X_{1}^{k}=y_{1}^{k}|S_{1}^{n}=na_{n})$ is invariant under any i.i.d sampling scheme in the family of densities $\pi^{\alpha}$ as $\alpha$ belongs to $Im(X_{1})$ (commonly called tilting change of measure).
Namely
\[
p(X_{1}^{k}=y_{1}^{k}|S_{1}^{n}=na_{n})=\pi^{\alpha}(X_{1}^{k}=y_{1}^{k}|S_{1}^{n}=na_{n})
\]
where on the LHS the $X_{i}$'s are sampled i.i.d. under $p$ and on the RHS they are sampled i.i.d. under $\pi^{\alpha}$.

Using Bayes formula, it thus holds 
\begin{align}
& p(X_{1}=y_{1}|S_{1}^{n}=na_{n})=\pi ^{m}(X_{1}=y_{1}|S_{1}^{n}=na_{n}) \notag \\
& =\pi ^{m}(X_{1}=y_{1})\frac{\pi ^{m}(S_{2}^{n}=na_{n}-y_{1})}{\pi^{m}(S_{1}^{n}=na_{n})}  \notag \\
& =\frac{\sqrt{n}}{\sqrt{n-1}}\pi ^{m}(X_1=y_1)\frac{\widetilde{\pi_{n-1}}(\frac{m-y_{1}}{s\sqrt{n-1}})}{\widetilde{\pi _{n}}(0)}, \label{Bayes formula}
\end{align}
where $\widetilde{\pi_{n-1}}$ is the normalized density of $S_{2}^{n}$ under i.i.d. sampling with the density $\pi ^{a_{n}};$ correspondingly, $\widetilde{\pi_{n}}$ is the normalized density of $S_{1}^{n}$ under the same sampling. Note that a r.v. with density $\pi ^{an}$ has expectation $m$ and variance $s^{2}$. Perform a third-order Edgeworth expansion of $\widetilde{\pi_{n-1}}(z)$, using Theorem \ref{3theorem1}. It follows 
\[
\widetilde{\pi_{n-1}}(z)=\phi(z)\Big(1+\frac{\mu_{3}}{6s^{3}\sqrt{n-1}}(z^{3}-3z)\Big)+o\Big(\frac{1}{\sqrt{n}}\Big),
\]
The approximation of $\widetilde{\pi_{n}}(0)$ is 
\[
\widetilde{\pi_{n}}(0)=\phi(0)\Big(1+o\big(\frac{1}{\sqrt{n}}\big)\Big).
\]
Hence (\ref{Bayes formula}) becomes 
\begin{align}
& p(X_{1}=y_{1}|S_{1}^{n}=na_{n})  \notag \\
& =\frac{\sqrt{n}}{\sqrt{n-1}}\pi^{m}(X_{1}=y_{1})\frac{\phi(z)}{\phi(0)}\Big[1+\frac{\mu_{3}}{6s^{3}\sqrt{n-1}}(z^{3}-3z)+o\Big(\frac{1}{\sqrt{n}}\Big)\Big] \label{conddens} \\
& =\frac{\sqrt{2\pi n}}{\sqrt{n-1}}\pi^{m}(X=y_{1}){\phi(z)}\big(1+R_{n}+o(1/\sqrt{n})\big),  \notag
\end{align}
where 
\[
R_{n}=\frac{\mu_{3}}{6s^{3}\sqrt{n-1}}(z^{3}-3z).
\]
Under condition (\ref{croissance de aen o(1)}), by Corollary (\ref{3cor1}), $\mu_{3}/s^{3}\rightarrow 0.$ This yields 
\[
R_{n}=o\big(1/\sqrt{n}\big),
\]
which gives
\[
p(X_{1}=y_{1}|S_{1}^{n}=na_{n})=\pi^{m}(X=y_{1})\big(1+o(1/\sqrt{n})\big)
\]
as claimed.

\subsection{Proof of Proposition \protect\ref{Prop indCondCroissance Lente}}

Denote 
\[
z_{i}:=\frac{m_i-y_{i+1}}{s_{i}\sqrt{n-i-1}}
\]
where 
\[
s_{i}^{2}:=s^{2}(t_{i}).
\]

We first state a Lemma pertaining to the order of magnitude of $z_{i}$. The proof of this Lemma is in the next Subsection

\begin{lemma}
\label{3lemma z} Assume that $h(x)\in\mathcal{R}$. Let $t_{i}$ be defined by (\ref{def-t_i}). Assume that $a_{n}\rightarrow\infty$ as $n\rightarrow\infty$ and that (\ref{croissance de aen o(1)}) holds. Then as $n\rightarrow\infty$ 
\[
\lim_{n\rightarrow\infty}\sup_{0\leq i\leq k-1}z_{i}=0,\qquad\emph{and}\qquad\sup_{0\leq i\leq k-1}z_{i}^{2}=o\left(\frac{1}{\sqrt{n}}\right).
\]
\end{lemma}

We turn to the proof of Proposition \ref{Prop indCondCroissance Lente}.

It holds by Bayes formula, 
\[
p_{a_{n}}(y_{1}^{k})=\prod_{i=0}^{k-1}p(X_{i+1}=y_{i+1}|S_{i+1}^{n}=na_{n}-s_{1}^{i}).
\]

Using the invariance of the conditional distributions under the tilting it holds, for any $i$ between $1$ and $k-1$
\begin{align*}
& p(X_{i+1}=y_{i+1}|S_{i+1}^{n}=na_{n}-S_{1}^{i})=\frac{\sqrt{2\pi (n-i)}}{\sqrt{n-i-1}}\pi^{m_{i}}(X_{i+1}=y_{i+1}){\phi(z_{i})}\big(1+o(1/\sqrt{n})\big) \\
& =\frac{\sqrt{n-i}}{\sqrt{n-i-1}}\pi^{m_{i}}(X_{i+1}=y_{i+1}){\big(1-z_{i}^{2}/2+o(z_{i}^{2})\big)}\left(1+o(1/\sqrt{n})\right)\big),
\end{align*}
where we used a Taylor expansion in the second equality. Using once more Lemma $\ref{3lemma z}$, under conditions (\ref{croissance de aen o(1)}), we have as $a_{n}\rightarrow\infty$ 
\[
z_{i}^{2}=o(1/\sqrt{n}).
\]%
Hence we get 
\[
p(X_{i+1}=y_{i+1}|S_{i+1}^{n}=na_{n}-s_{1}^{i})=\frac{\sqrt{n-i}}{\sqrt{n-i-1}}\pi^{m_{i}}(X_{i+1}=y_{i+1})\big(1+o(1/\sqrt{n})\big),
\]
which yields 
\begin{align*}
p(X_{1}^{k}=y_{1}^{k}|S_{1}^{n}=na_{n})& =\prod_{i=0}^{k-1}\Big(\frac{\sqrt{n-i}}{\sqrt{n-i-1}}\pi^{m_{i}}(X_{i+1}=y_{i+1})\big(1+o(1/\sqrt{n})\big)\Big) \\
& =\prod_{i=0}^{k-1}\pi^{m_{i}}(X_{i+1}=y_{i+1})\prod_{i=0}^{k-1}\frac{\sqrt{n-i}}{\sqrt{n-i-1}}\prod_{i=0}^{k-1}\left(1+o\left( \frac{1}{\sqrt{n}}\right)\right) \\
& =\left(1+o\left(\frac{1}{\sqrt{n}}\right)\right)\prod_{i=0}^{k-1}\pi^{m_{i}}(X_{i+1}=y_{i+1}),
\end{align*}
The proof is completed.

\subsection{Proof of Lemma \protect\ref{3lemma z}}

When $n\rightarrow\infty$, it holds 
\[
z_{i}\sim m_{i}/(s_{i}\sqrt{n}).
\]
From Theorem \ref{THM3.1}, it holds 
\[
z_{i}\sim\frac{\psi(t_{i})}{\sqrt{n\psi^{^{\prime}}(t_{i})}}.
\]
Since $m_{i}\sim m_{k}$ as $n\rightarrow\infty$, it holds
\[
m_{i}\sim\psi(t_{k}).
\]
Hence 
\[
\psi(t_{i})\sim\psi(t_{k}).
\]
\textbf{Case 1:} if $h(x)\in R_{\beta}$. Hence 
\[
h^{^{\prime}}(x)=x^{\beta-1}l_{0}(x)\left(\beta +\epsilon(x)\right).
\]
Set $x=\psi(t)$; we get 
\[
h^{^{\prime}}\left(\psi(t)\right)=\psi(t)^{\beta-1}l_{0}\left(\psi
(t)\right)\left(\beta+\epsilon\left(\psi(t)\right)\right).
\]
Notice that $\psi^{^{\prime}}(t)=1/h^{^{\prime}}\left(\psi(t)\right)$;
we obtain 
\[
\frac{\psi^{\prime}(t_{i})}{\psi ^{\prime}(t_{k})}=\frac{h^{^{\prime}}\left(\psi(t_{k})\right)}{h^{^{\prime}}\left(\psi (t_{i})\right)}=\frac{\left(\psi(t_{k})\right)^{\beta-1}l_{0}\left(\psi(t_{k})\right)\left(\beta+\epsilon\left(\psi(t_{k})\right)\right)}{\left(\psi(t_{i})\right)^{\beta-1}l_{0}\left(\psi(t_{i})\right)\left(\beta+\epsilon\left(\psi(t_{i})\right)\right)}\longrightarrow 1,
\]
where we use the slowly varying propriety of $l_{0}$. Thus it holds 
\[
\psi^{\prime}(t_{i})\sim\psi^{\prime}(t_{k}),
\]
which yields 
\[
z_{i}\sim\frac{\psi(t_{k})}{\sqrt{n\psi^{^{\prime}}(t_{k})}}.
\]
Hence we have under condition (\ref{croissance de aen o(1)}) 
\[
z_{i}^{2}\sim\frac{\psi(t_{k})^{2}}{{n\psi^{^{\prime}}(t_{k})}}=\frac{\psi(t_{k})^{2}}{{\sqrt{n}\psi^{^{\prime}}(t_{k})}}\frac{1}{\sqrt{n}}=o\left(\frac{1}{\sqrt{n}}\right) ,
\]
which implies further that $z_{i}\rightarrow 0$.

\textbf{Case 2:} if $h(x)\in R_{\infty}$. It holds $m(t_{k})\geq m(t_{i})$ as $n\rightarrow\infty$. Since the function $t\rightarrow m(t)$ is increasing, we have 
\[
t_{i}\leq t_{k}.
\]
The function $t\rightarrow\psi^{^{\prime}}(t)$ is decreasing, since 
\[
\psi^{^{\prime\prime}}(t)=-\frac{\psi(t)}{t^{2}}\epsilon(t)\left(1+o(1)\right)<0\qquad as\quad t\rightarrow\infty.
\]
Therefore as $n\rightarrow\infty$ 
\[
\psi^{\prime}(t_{i})\geq\psi^{\prime}(t_{k})>0,
\]
which yields 
\[
z_{i}\sim\frac{\psi(t_{i})}{\sqrt{n\psi^{\prime}(t_{i})}}\leq\frac{2\psi (t_{k})}{\sqrt{n\psi^{\prime}(t_{k})}},
\]
hence we have 
\[
z_{i}^{2}\leq\frac{4\psi(t_{k})^{2}}{{n\psi^{\prime}(t_{k})}}=\frac{4\psi(t_{k})^{2}}{{\sqrt{n}\psi^{\prime}(t_{k})}}\frac{1}{\sqrt{n}}=o\left(\frac{1}{\sqrt{n}}\right),
\]
where the last step holds from condition (\ref{croissance de aen o(1)}). Further it holds $z_{i}\rightarrow 0$.

This closes the proof of the Lemma.

\subsection{Proof of Lemma \protect\ref{Lemme s^2 self neglecting}}

\textbf{Case 1:} if $h(t)\in R_{\beta }$. By Theorem \ref{THM3.1}, it holds $s^{2}\sim\psi^{\prime}(t)$ with $\psi(t)\sim t^{1/\beta}l_{1}(t)$, where $l$ is some slowly varying function. Consider $\psi^{\prime}(t)=1/h^{^{\prime}}\big(\psi(t)\big)$, hence  
\begin{align*}
\frac{1}{s^{2}}& \sim h^{^{\prime}}\big(\psi(t)\big)=\psi(t)^{\beta-1}l_{0}\big(\psi(t)\big)\big(\beta+\epsilon\big(\psi(t)\big)\big) \\
& \sim\beta t^{1-1/\beta}l_{1}(t)^{\beta-1}l_{0}\big(\psi(t)\big)=o(t),
\end{align*}
where $l_{0}\in R_{0}$. This implies for any $u\in K$ 
\[
\frac{u}{s}=o(\sqrt{t}),
\]%
which  yields, using (\ref{2.3}) 
\begin{align*}
\frac{s^{2}\left(t+u/s\right)}{s^{2}}& \sim\frac{\psi^{\prime}(t+u/s)}{\psi^{\prime}(t)}=\frac{\psi(t)^{\beta-1}l_{0}\big(\psi(t)\big)\big(\beta+\epsilon\big(\psi(t)\big)\big)}{\big(\psi(t+u/s)\big)^{\beta-1}l_{0}\big(\psi(t+u/s)\big)\big(\beta+\epsilon\big(\psi(t+u/s)\big)\big)} \\
& \sim\frac{\psi(t)^{\beta-1}}{\psi(t+u/s)^{\beta-1}}\sim\frac{t^{1-1/\beta}l_{1}(t)^{\beta-1}}{(t+u/s)^{1-1/\beta}l_{1}(t+u/s)^{\beta-1}}\longrightarrow 1.
\end{align*}
\textbf{Case 2:} if $h(t)\in R_{\infty}$. Then $\psi(t)\in\widetilde{R_{0}}$, hence it holds 
\[
\frac{1}{st}\sim\frac{1}{t\sqrt{\psi^{\prime}(t)}}=\sqrt{\frac{1}{t\psi(t)\epsilon(t)}}\longrightarrow 0,
\]
which last step holds from condition (\ref{2.4}). Hence for any $u\in K$, we get as $n\rightarrow\infty$ 
\[
\frac{u}{s}=o(t),
\]
thus using the slowly varying propriety of $\psi(t)$ we have 
\begin{align}
\frac{s^{2}\left(t+u/s\right)}{s^{2}}& \sim\frac{\psi^{\prime}(t+u/s)}{\psi^{\prime}(t)}=\frac{\psi(t+u/s)\epsilon(t+u/s)}{t+u/s}\frac{t}{\psi(t)\epsilon(t)}  \notag \\
& \sim\frac{\epsilon(t+u/s)}{\epsilon(t)}=\frac{\epsilon(t)+O\big(\epsilon^{\prime}(t)u/s\big)}{\epsilon(t)}\longrightarrow 1,
\label{3fqdq}
\end{align}
where we used a Taylor expansion in the second line, and where the last step holds from condition (\ref{2.4}). This completes the proof.

\subsection{Proof of Theorem \protect\ref{point conditional density e}}

Making use of 
\[
p(X_{1}^{k}=y_{1}^{k}|S_{1}^{n}=na_{n})=\prod_{i=0}^{k-1}p(X_{i+1}=y_{i+1}|S_{i+1}^{n}=na_{n}-s_{1}^{i}),
\]
and using the tilted density $\pi^{a_{n}}$ instead of $\pi^{m_{i}}$ it holds 
\begin{equation}
p(X_{i+1}=y_{i+1}|S_{i+1}^{n}=na_{n}-s_{1}^{i})=\frac{\sqrt{n-i}}{\sqrt{n-i-1}}\pi^{a_{n}}(X_{i+1}=y_{i+1})\frac{\widetilde{\pi_{n-i-1}}(\frac{(i+1)a_{n}-s_{1}^{i+1}}{s\sqrt{n-i-1}})}{\widetilde{\pi _{n-i}}\left(\frac{ia_{n}-s_{1}^{i}}{s\sqrt{n-i}}\right)},  \label{bayes formula e}
\end{equation}
where $\widetilde{\pi_{n-i-1}}$ is the normalized density of $S_{i+2}^{n}$ under i.i.d. sampling with $\pi^{a_{n}}$. Correspondingly, denote $\widetilde{\pi_{n-i}}$ the normalized density of $S_{i+1}^{n}$ under the same sampling. Write 
\[
z_{i}=\frac{ia_{n}-s_{1}^{i-1}}{s\sqrt{n-i+1}}.
\]%
By Theorem \ref{3theorem1} a third-order Edgeworth expansion yields 
\[
\widetilde{\pi_{n-i-1}}(z_{i})=\phi(z_{i})\left(1+R_{n}^{i}\right)+o\left(\frac{1}{\sqrt{n}}\right),
\]
where 
\[
R_{n}^{i}=\frac{\mu_{3}}{6s^{3}\sqrt{n-i-1}}(z_{i}^{3}-3z_{i}).
\]%
Accordingly 
\[
\widetilde{\pi_{n-i}}(z_{i-1})=\phi(z_{i-1})\left(1+R_{n}^{i-1}\right)+o\left(\frac{1}{\sqrt{n}}\right).
\]
When $a_{n}\rightarrow\infty$, using Theorem \ref{THM3.1}, it holds 
\begin{align}
\sup_{0\leq i\leq k-1}z_{i}^{2}\sim\frac{(i+1)^{2}a_{n}^{2}}{s^{2}{n}}& \leq\frac{2k^{2}a_{n}^{2}}{s^{2}{n}}=\frac{2k^{2}(m(t))^{2}}{s^{2}{n}} \notag \\
& \sim\frac{2k^{2}(\psi(t))^{2}}{\psi^{\prime}(t){n}}=\frac{2k^{2}(\psi(t))^{2}}{\sqrt{n}\psi^{\prime}(t)}\frac{1}{\sqrt{n}}=o\left( \frac{1}{\sqrt{n}}\right),  \label{ordre de z_i}
\end{align}
where the last step holds under condition (\ref{croissance de aen o(1)}).
Hence it holds $z_{i}\rightarrow 0$ for $0\leq i\leq k-1$ as $a_{n}\rightarrow\infty$, and by Corollary \ref{3cor1}, $\mu_{3}/s^{3}\rightarrow 0$; Hence 
\[
R_{n}^{i}=o\left(1/\sqrt{n}\right)\text{ and }R_{n}^{i-1}=o\left(1/\sqrt{n}\right).
\]
We thus get 
\begin{align*}
& p(X_{i+1}=y_{i+1}|S_{i+1}^{n}=na_{n}-s_{1}^{i})=\frac{\sqrt{n-i}}{\sqrt{n-i-1}}\pi^{a_{n}}(X_{i+1}=y_{i+1})\frac{\phi(z_{i})}{\phi(z_{i-1})}\left(1+o(1/\sqrt{n})\right) \\
& =\frac{\sqrt{n-i}}{\sqrt{n-i-1}}\pi^{a_{n}}(X_{i+1}=y_{i+1})\left( {1-(z_{i}^{2}-z_{i-1}^{2})/2+o(z_{i}^{2}-z_{i-1}^{2})}\right)\left( 1+o(1/\sqrt{n})\right),
\end{align*}
where we used a Taylor expansion in the second equality. Using (\ref{ordre de z_i}), we have as $a_{n}\rightarrow\infty$ 
\[
|z_{i}^{2}-z_{i-1}^{2}|=o(1/\sqrt{n}),
\]
from which 
\[
p(X_{i+1}=y_{i+1}|S_{i+1}^{n}=na_{n}-s_{1}^{i})=\frac{\sqrt{n-i}}{\sqrt{n-i-1}}\pi ^{a_{n}}(X_{i+1}=y_{i+1})\left(1+o(1/\sqrt{n})\right),
\]
which yields
\begin{align*}
p(X_{1}^{k}=y_{1}^{k}|S_{1}^{n}=na_{n})& =\prod_{i=0}^{k-1}\left( \pi
^{a_{n}}(X_{i+1}=y_{i+1})\sqrt{\frac{n}{n-k}}\right)\prod_{i=0}^{k-1}\left(1+o\left(\frac{1}{\sqrt{n}}\right)\right) \\
& =\left(1+o\left(\frac{1}{\sqrt{n}}\right)\right)\prod_{i=0}^{k-1}\pi^{a_{n}}(X_{i+1}=y_{i+1}).
\end{align*}
This completes the proof.

\subsection{Proof of Lemma \protect\ref{JensenLemme}}

For a density $p(x)$ defined in as in (\ref{2.1}), we show that $g(x)$ is a convex function when $x$ is large. If $h(x)\in R_{\beta}$, for $x$ large 
\[
g^{^{\prime\prime}}(x)=h^{^{\prime}}(x)=\frac{h(x)}{x}\left(\beta
+\epsilon(x)\right)>0.
\]
If $h(x)\in R_{\infty}$, its reciprocal function $\psi(x)\in \widetilde{R_{0}}$. Set $x:=\psi(v)$. Then 
\[
g^{^{\prime\prime}}(x)=h^{\prime}(x)=\frac{1}{\psi^{\prime }(v)}=\frac{v}{\psi(v)\epsilon(v)}>0,
\]
where the inequality holds since $\epsilon(v)>0$ when $v$ is large enough. Hence $g(x)$ is convex for large $x$. Therefore, the density $p(x)$ with  $h(x)\in \mathcal{R}$ satisfies the conditions of Theorem 6.2.1 in \cite{Jensen}. Denote by $p_{n}$ the density of $\bar{X}=(X_{1}+\ldots+X_{n})/n$. We obtain from formula $(2.2.6)$ of \cite{Jensen}, using a third order Edgeworth expansion 
\[
P(S_{1}^{n}\geq na_{n})=\frac{\Phi(t)^{n}\exp(-nta_{n})}{\sqrt{n}ts(t)}\left(B_{0}(\lambda_{n})\right)+O\left(\frac{\mu_{3}(t)}{6\sqrt{n}s^{3}(t)}B_{3}(\lambda_{n})\right),
\]
where $\lambda_{n}=\sqrt{n}ts(t)$, $B_{0}(\lambda_{n})$ and $B_{3}(\lambda_{n})$ are defined by 
\[
B_{0}(\lambda_{n})=\frac{1}{\sqrt{2\pi}}\left(1-\frac{1}{\lambda_{n}^{2}}+o(\frac{1}{\lambda_{n}^{2}})\right),\qquad B_{3}(\lambda_{n})\sim-\frac{3}{\sqrt{2\pi}\lambda_{n}}.
\]
We show that as $a_{n}\rightarrow\infty$ 
\begin{equation}
\frac{1}{\lambda_{n}^{2}}=o\left(\frac{1}{n}\right).  \label{F1}
\end{equation}
Since $n/\lambda_{n}^{2}=1/(t^{2}s^{2}(t))$, (\ref{F1}) is equivalent to show that 
\[
t^{2}s^{2}(t)\longrightarrow\infty.
\]%
By Theorem \ref{THM3.1}, $m(t)\sim\psi(t)$ and $s^{2}(t)\sim\psi^{\prime}(t)$; combined with $m(t)=a_{n}$, it holds $t\sim h(a_{n})l_{1}(a_{n})$, where $l_{1}$ is some slowly varying function. If $h\in R_{\beta}$, notice that 
\[
\psi^{\prime}(t)=\frac{1}{h^{\prime}(\psi(t))}=\frac{\psi(t)}{h\left(
\psi(t)\right)\left(\beta+\epsilon(\psi(t))\right)}\sim\frac{a_{n}}{h(a_{n})\left(\beta +\epsilon(\psi(t))\right)};
\]
hence 
\[
t^{2}s^{2}(t)\sim h(a_{n})^{2}l_{1}(a_{n})^{2}\frac{a_{n}}{h(a_{n})\left(\beta+\epsilon(\psi(t))\right)}=\frac{a_{n}h(a_{n})l_{1}(a_{n})^{2}}{\beta+\epsilon(\psi(t_{n}))}\longrightarrow\infty.
\]
If $h\in R_{\infty}$, then $\psi(t)\in\widetilde{R_{0}}$, thus 
\[
t^{2}s^{2}(t)\sim t^{2}\frac{\psi(t)\epsilon(t)}{t}=t\psi(t)\epsilon
(t)\longrightarrow\infty ,
\]
Summing up we have proved that  
\[
B_{0}(\lambda_{n})=\frac{1}{\sqrt{2\pi}}\left(1+o\left(\frac{1}{n}\right)\right).
\]
By (\ref{F1}), $\lambda_{n}$ goes to $\infty$ as $a_{n}\rightarrow \infty $; this implies further that $B_{3}(\lambda_{n})\rightarrow 0$. On the other hand, by Corollary \ref{3cor1} it holds $\mu_{3}/s^{3}\rightarrow 0$. Hence we obtain  
\[
P(S_{1}^{n}\geq na_{n})=\frac{\Phi(t)^{n}\exp(-nta_{n})}{\sqrt{2\pi n}ts(t)}\left(1+o\left(\frac{1}{\sqrt{n}}\right)\right),
\]
which gives (\ref{Hoglund}). By Jensen's Theorem 6.2.1 (\cite{Jensen}) and formula (2.2.4) in \cite{Jensen} it follows uniformly in $\tau$ 
\[
p(S_{1}^{n}/n=\tau)=\frac{\sqrt{n}\Phi(t_{\tau})^{n}\exp(-nt_{\tau}\tau)}{\sqrt{2\pi }s(t_{\tau})}\left(1+o\left(\frac{1}{\sqrt{n}}\right)\right),
\]
which, together with $p(S_{1}^{n}=n\tau)=(1/n)p(S_{1}^{n}/n=\tau)$, gives (\ref{Hoglund density}).

\subsection{Proof of Theorem \ref{main theorem}}
It holds 
\begin{align}
p_{A_{n}}(y_{1})& =\int_{a_{n}}^{\infty}p(X_{1}=y_{1}|S_{1}^{n}=n\tau)p(S_{1}^{n}=n\tau|S_{1}^{n}\geq na_{n})d\tau   \notag \\
& =\frac{p(X_{1}=y_{1})}{P(S_{1}^{n}\geq na_{n})}\int_{a_{n}}^{\infty}p(S_{2}^{n}=n\tau-y_{1})d\tau   \notag \\
& =\left(1+\frac{P_{2}}{P_{1}}\right)\frac{p(X_{1}=y_{1})}{P(S_{1}^{n}\geq na_{n})}\int_{a_{n}}^{a_{n}+\eta_{n}}p(S_{2}^{n}=n\tau-y_{1})d\tau  \notag\\
& =\left(1+\frac{P_{2}}{P_{1}}\right)\int_{a_{n}}^{a_{n}+\eta_{n}}p(X_{1}=y_{1}|S_{1}^{n}=n\tau)p(S_{1}^{n}=n\tau|S_{1}^{n}\geq na_{n})d\tau   \label{forme eceedance}
\end{align}
where the second equality is obtained by Bayes formula, and 
\[
P_{1}=\int_{a_{n}}^{a_{n}+\eta_{n}}p(S_{2}^{n}=n\tau-y_{1})d\tau ,
\]
\[
P_{2}=\int_{a_{n}+\eta_{n}}^{\infty}p(S_{2}^{n}=n\tau-y_{1})d\tau.
\]
We show that $P_{2}$ is infinitely small with respect to $P_{1}$. Indeed 
\begin{align*}
& P_{2}=\frac{1}{n}P\left(S_{2}^{n}\geq {n(a_{n}+\eta_{n})-y_{1}}\right)=\frac{1}{n}P\left(S_{2}^{n}\geq (n-1)c_{n}\right), \\
& P_{1}+P_{2}=\frac{1}{n}P\left(S_{2}^{n}\geq {na_{n}-y_{1}}\right) =\frac{1}{n}P\left(S_{2}^{n}\geq(n-1)d_{n}\right),
\end{align*}
where $c_{n}=\left(n(a_{n}+\eta_{n})-y_{1}\right)/(n-1)$ and $d_{n}=(na_{n}-y_{1})/(n-1)$. Denote $t_{c_{n}}=m^{-1}(c_{n})$ and $t_{d_{n}}=m^{-1}(d_{n})$. Using Lemma \ref{JensenLemme}, it holds 
\[
\frac{P_{2}}{P_{1}+P_{2}}=\left(+o\left(\frac{1}{\sqrt{n}}\right)\right)\frac{t_{d_{n}}s(t_{d_{n}})}{t_{c_{n}}s(t_{c_{n}})}\exp\left(-(n-1)\left(I(c_{n})-I(d_{n})\right)\right).
\]
Using the convexity of the function $I$, it holds 
\begin{align*}
\exp\left(-(n-1)I(c_{n})-I(d_{n})\right)& \leq\exp-(n-1)(c_{n}-d_{n})m^{-1}(d_{n})\\
& =\exp -n\eta_{n}m^{-1}(d_{n}).
\end{align*}
The function $u\rightarrow m^{-1}(u)$ is increasing. Since $d_{n}\geq a_{n}$ as $a_{n}\rightarrow\infty$, it holds $m^{-1}(d_{n})\geq m^{-1}(a_{n})$; hence $\exp-(n-1)\left(I(c_{n})-I(d_{n})\right)\leq\exp-n\eta_{n}m^{-1}(a_{n})\longrightarrow 0$. We now show that 
\[
\frac{t_{d_{n}}s(t_{d_{n}})}{t_{c_{n}}s(t_{c_{n}})}\longrightarrow 1.
\]
By definition, $c_{n}/d_{n}\rightarrow 1$ as $a_{n}\rightarrow\infty$. If $h\in R_{\beta}$, it holds 
\[
\left(\frac{t_{d_{n}}s(t_{d_{n}})}{t_{c_{n}}s(t_{c_{n}})}\right)^{2}\sim\left(\frac{d_{n}h(d_{n})}{\beta+\epsilon\left(\psi(d_{n})\right)}\right)^{2}\left(\frac{\beta+\epsilon\left(\psi(c_{n})\right)}{c_{n}h(c_{n})}\right)^{2}\sim\left(\frac{h(d_{n})}{h(c_{n})}\right)^{2}\longrightarrow 1.
\]
If $h\in R_{\infty}$, 
\[
t^{2}s^{2}(t)\sim t\psi(t)\epsilon(t),
\]
hence 
\[
\left(\frac{t_{d_{n}}s(t_{d_{n}})}{t_{c_{n}}s(t_{c_{n}})}\right)^{2}\sim\frac{d_{n}\psi(d_{n})\epsilon(d_{n})}{c_{n}\psi(c_{n})\epsilon(c_{n})}\sim\frac{\epsilon(d_{n})}{\epsilon(c_{n})}=\frac{\epsilon\left(c_{n}-n\eta_{n}/(n-1)\right)}{\epsilon(c_{n})}\longrightarrow 1,
\]
where last step holds by using the same argument as in the second line of (\ref{3fqdq}). We obtain 
\[
\frac{P_{2}}{P_{1}}=o\left(1\right).
\]
Therefore $p_{A_{n}}(y_{1})$ can be approximated by 
\[
p_{A_{n}}(y_{1})=\left(1+o\left(1\right)\right)\int_{a_{n}}^{a_{n}+\eta_{n}}p(X_{1}=y_{1}|S_{1}^{n}=n\tau)p(S_{1}^{n}=n\tau|S_{1}^{n}\geq na_{n})d\tau.
\]
By Lemma \ref{JensenLemme}, it follows that uniformly when $\tau\in\lbrack a_{n},a_{n}+\eta_{n}]$ 
\begin{align}
& p(S_{1}^{n}=n\tau|S_{1}^{n}\geq na_{n})=\frac{p(S_{1}^{n}=n\tau)}{P(S_{1}^{n}\geq na_{n})}  \notag \\
& =\left(1+o\left(\frac{1}{\sqrt{n}}\right)\right)\frac{ts(t)}{s(t_{\tau})}\exp\left(-n(I(\tau)-I(a_{n}))\right),  \label{finpreuve}
\end{align}
We now turn back to (\ref{forme eceedance}) and note that under the appropriate condition (\ref{croissance de aen o(1)}) or (\ref{croissance de aen O(1)})\ the corresponding approximating density $\pi^{\tau }$ or $g_{\tau}$ can be seen to hold uniformly on $\tau $ in $(a_{n},a_{n}+\eta_{n})$. Inserting (\ref{finpreuve}) into (\ref{forme eceedance}), we complete the proof of Theorem \ref{main theorem} insering the corresponding local result.

\end{document}